\documentclass[11pt]{amsart}
\usepackage{amssymb}
\usepackage{graphicx}
\usepackage{xcolor} % A package to add color.
\usepackage{fullpage} % Sets all margins to 1 in.  
\usepackage{amsmath}
\usepackage{amsthm}
\usepackage{mathtools}
\usepackage{verbatim}
\usepackage{hyperref}
\usepackage{enumitem}
\usepackage{eucal}
\usepackage{mathrsfs}
\usepackage{upgreek}

\usepackage{tikz}
\usepackage{pgfplots}
\usepgfplotslibrary{fillbetween}
\usetikzlibrary{patterns}

\definecolor{green}{rgb}{0,0.8,0} % Redefines the color green.
\newtheorem{theorem}{Theorem}[section]
\newtheorem{corollary}[theorem]{Corollary}
\newtheorem{lemma}[theorem]{Lemma}
\newtheorem{proposition}[theorem]{Proposition}

\theoremstyle{definition}
\newtheorem{definition}[theorem]{Definition}

\theoremstyle{remark}
\newtheorem{remark}[theorem]{Remark}

\numberwithin{equation}{section}
\newcommand{\nnrm}[1]{{\vert\kern-0.25ex\vert\kern-0.25ex\vert #1 
		\vert\kern-0.25ex\vert\kern-0.25ex\vert}}

\newcommand{\supp}{{\mathrm{supp}}\,}

\newcommand{\bbN}{\mathbb N}

\newcommand{\bbR}{\mathbb R}

\newcommand{\bbV}{\mathbb V}

\newcommand{\calB}{\mathcal B}
\newcommand{\calC}{\mathcal C}
\newcommand{\calD}{\mathcal D}
\newcommand{\calE}{\mathcal E}
\newcommand{\calF}{\mathcal F}

\newcommand{\calL}{\mathscr L}

\newcommand{\calP}{\mathcal P}

\newcommand{\frkR}{\mathfrak R}

\newcommand{\frkX}{\mathfrak X}

\newcommand{\sfF}{\mathsf{F}}

\newcommand{\sfc}{\mathsf{c}}

\DeclareMathOperator*{\argmin}{argmin}
\DeclareMathOperator*{\esssup}{esssup}

\newcommand{\R}{\mathbb{R}}
\newcommand{\N}{\mathbb N}

\newcommand{\mc}{\mathcal C}

\newcommand{\bq}{\begin{equation}}
	\newcommand{\eq}{\end{equation}}
\newcommand{\e}{\varepsilon}
\newcommand{\lt}{\left}
\newcommand{\rt}{\right}

\newcommand{\pa}{\partial}

\newcommand{\intr}{\int_{\bbR^d}}
\newcommand{\intrr}{\iint_{\bbR^d \times \bbR^d}}

	\author{Young-Pil Choi}
	\address{Young-Pil Choi, Department of Mathematics, Yonsei University, Seoul 03722, Republic of Korea}
	\email{ypchoi@yonsei.ac.kr}
	\author{Jeongho Kim}
	\address{Jeongho Kim, Department of Applied Mathematics, Kyung Hee University, 1732 Deogyeong-Daero, Giheung-gu, Yongin-si, Gyeonggi-Do 17104, Republic of Korea}
	\email{jeonghokim@khu.ac.kr}
	\author{Oliver Tse}
	\address{Oliver Tse, Department of Mathematics and Computer Science, Eindhoven University of Technology, 5600 MB Eindhoven, the Netherlands}
	\email{o.t.c.tse@tue.nl}

\begin{document}
	\bibliographystyle{plain}
	\title{Derivation of nonlinear aggregation-diffusion equation from a kinetic BGK-type equation}
	
	\date\today
	
	\maketitle

	\renewcommand{\thefootnote}{\fnsymbol{footnote}}
	\footnotetext{
%		\emph{2020 AMS Mathematics Subject Classification:}  
	}
	\renewcommand{\thefootnote}{\arabic{footnote}}

	\begin{abstract} 
	This paper investigates the diffusion limit of a kinetic BGK-type equation, focusing on its relaxation to a nonlinear aggregation-diffusion equation, where the diffusion exhibits a porous-medium-type nonlinearity. Unlike previous studies by Dolbeault et al. [Arch. Ration. Mech. Anal., 186, (2007), 133--158] and Addala and Tayeb [J. Hyperbolic Differ. Equ., 16, (2019), 131--156], which required bounded initial data, our work considers initial data that need not be bounded. We develop new techniques for handling weak entropy solutions that satisfy the natural bounds associated with the kinetic entropy inequality. Our proof employs the relative entropy method and various compactness arguments to establish the convergence and properties of these solutions.
	\end{abstract}
	
% \tableofcontents
	%%%%%%%%%%%%%%%%%%%%%%%%%%%%%%%%%%%%%%%%%%%%%%
	%
	%
	%
	%
	%
	%
	%%%%%%%%%%%%%%%%%%%%%%%%%%%%%%%%%%%%%%%%%%%%%%

	\section{Introduction}
The primary objective of this work is to justify nonlinear aggregation-diffusion equations as limits of a suitable simple kinetic model. Specifically, we rigorously derive the  equation
	\bq\label{main_eq}
	\pa_t \rho = \Delta_x \rho^\gamma + \nabla_x \cdot (\rho \sfF[\rho]), \qquad \gamma \in \lt(1, 1 + \frac2{d+2} \rt],%\tag{\textsf{PME}}
	\eq
	where the unknown $\rho = \rho(t,x)$ represents the density at $(t,x) \in \R_+ \times \bbR^d$ with $d \geq 1$ and $\sfF[\rho]$ is a (nonlocal) force of the form
	\[
		\sfF[\rho] = \nabla V + \nabla K\ast\rho.
	\]
	Here, $V\colon\bbR^d \to \R_+$ is a given external potential, and $K\colon \bbR^d \to \R$ is an interaction potential. Equations like the ones above appear in various contexts, including phase separation of lattice active matter, granular flow, biological swarming, and pattern formation \cite{CCCSS16, CMV03, GL1, GL2, ME99}. For a thorough discussion on nonlinear drift-diffusion models, we refer to \cite{CCY19, Va17} and the references therein.
	
To derive equation \eqref{main_eq} from a kinetic formulation, we follow an approach inspired by \cite{Bou99} and consider the following scaled BGK-type equation
	\bq\label{main_kin}
	\e \pa_t f^\e +  v \cdot \nabla_x f^\e - \sfF[\rho_{f^\e}] \cdot \nabla_v f^\e = \frac{1}{\e} Q_\gamma[f^\e],
	\eq
	where $f^\e= f^\e(t,x,v)$ is the one-particle distribution function at time $t \in \R_+$ and the phase-space point $(x,v) \in \bbR^d \times \bbR^d$. The local particle density $\rho_{f^\e}$ is given by
	\[
	\rho_{f^\e}(t,x) \coloneqq \intr f^\e(t,x,v)\,dv,
	\]
	and the BGK-type relaxation operator $Q_\gamma[f]=Q_\gamma[f](x,v)$ takes the form:
	\[
	Q_\gamma[f] \coloneqq  M_\gamma[\rho_f] - f,
	\]
	with the local equilibrium distribution function $M_\gamma$ defined as
		\bq\label{pre}
	M_\gamma[\rho](t,x,v) \coloneqq 
		\displaystyle c_{\gamma,d}\left(\frac{2\gamma}{\gamma-1}\rho^{\gamma-1}(t,x)-|v|^2\right)^{n/2}_+,   \qquad \gamma \in \lt(1, 1 + \frac2{d+2} \rt].
	\eq
	Here, the constants are defined as:
	\bq\label{para}
	n =\frac{2}{\gamma-1}-d\qquad \mbox{and} \qquad c_{\gamma,d} =\left(\frac{2\gamma}{\gamma-1}\right)^{-1/(\gamma-1)}\frac{\Gamma\left(\frac{\gamma}{\gamma-1}\right)}{\pi^{d/2}\Gamma(1+n/2)}. 
	\eq
    Since we consider the case when $1<\gamma\le 1+\frac{2}{d+2}$, we have $n\ge 2$ and $\gamma$ takes the form $\gamma = 1 + \frac{2}{n+d}$.

The equation \eqref{main_kin} can be rescaled by replacing the time variable with $\e t$ and by considering the relaxation time parameter $\tau = \e$, which is inversely related to the relaxation operator. Throughout this paper, we assume that $f^\e$ is a probability density, i.e., $\|f^\e(t)\|_{L^1} = 1$ for all $\e > 0$ and $t \geq 0$, since the total mass is preserved over time. 

	%%%%%%%%%%%%%%%%%%%%%%%%%%%%%%%%%%%%%%%%%%%%%%
	%
	%
	%
	%
	%
	%
	%%%%%%%%%%%%%%%%%%%%%%%%%%%%%%%%%%%%%%%%%%%%%%
%\subsection{Formal derivation}
\medskip

In the current work, we focus on the asymptotic analysis of the equation \eqref{main_kin} in the diffusion regime corresponding to the limit $\e \to 0$, which leads to the derivation of the equation \eqref{main_eq}. Here, we briefly outline the formal reasoning that explains why the equation \eqref{main_eq} is expected to be the limiting form of \eqref{main_kin} as $\e \to 0$. We begin by noting that direct computations, as detailed in \cite{CHpre}, yield the identities
	\begin{gather*}
	\intr M_\gamma[\rho](v)\,dv = \rho, \qquad \intr vM_\gamma[\rho](v)\,dv = 0,\\
	\intr v \otimes v\, M_\gamma[\rho](v)\,dv = \rho^\gamma \mathbb{I}_d, \qquad \gamma \in \lt(1, 1 + \frac2{d+2} \rt].
	\end{gather*}
	We then express the local moments (in $v$) of $f^\e$ to find
	\bq\begin{aligned}\label{eq:moments}
		&\e\pa_t \rho_{f^\e} + \nabla \cdot (m_{f^\e}) = 0,\cr
		&\e\pa_t m_{f^\e} +\nabla \cdot \intr v \otimes v\, f^\e\,dv + \rho_{f^\e} \sfF[\rho_{f^\e}]  = -\frac{1}{\e} m_{f^\e},
	\end{aligned}
	\eq
	where $m_f$ is the local momentum given by
	\[
	m_f \coloneqq \intr vf\,dv.
	\]
	Since $\e \ll 1$, it follows from the equation of $m_{f^\e}$ that
	\[
		-\frac1\e m_{f^\e} \approx \nabla \cdot \intr v \otimes v\, f^\e\,dv + \rho_{f^\e} \sfF[\rho_{f^\e}].
	\]
	Moreover, since $f^\e \approx M_\gamma[\rho_{f^\e}]$ for $\e \ll 1$, we have
	\[
	\nabla \cdot \intr v \otimes v\, f^\e\,dv \approx \nabla \cdot \intr v \otimes v\, M_\gamma[\rho_{f^\e}]\,dv = \nabla \rho_{f^\e}^\gamma,
	\]
	which ultimately then leads to 
	\[
		\pa_t \rho_{f^\e} = -\frac{1}{\e} \nabla \cdot (m_{f^\e})  \approx \Delta \rho_{f^\e}^\gamma + \nabla \cdot \bigl(\rho_{f^\e} \sfF[\rho_{f^\e}]\bigr) \qquad \mbox{for } \e \ll 1.
	\]

	%%%%%%%%%%%%%%%%%%%%%%%%%%%%%%%%%%%%%%%%%%%%%%
	%
	%
	%
	%
	%
	%
	%%%%%%%%%%%%%%%%%%%%%%%%%%%%%%%%%%%%%%%%%%%%%%
%\subsection{Literature review}
Regarding the diffusion limit of kinetic equations to the drift-diffusion equation, the linear diffusion case $(\gamma = 1)$ has been extensively studied. In this context, the linear Fokker--Planck operator $L_{\rm FP}(f) = \nabla_v \cdot (\nabla_v f + vf)$ is used on the right-hand side of the equation \eqref{main_kin}, instead of the BGK-type relaxation operator. A variational technique was proposed in \cite{DLPS17} to rigorously derive the limit from the Vlasov--Fokker--Planck equation to \eqref{main_eq} with $\gamma=1$ under sufficiently smooth regularity and integrability assumptions on the interaction potential $K$, for instance, globally bounded first and second derivatives. This approach relies on duality methods and compactness arguments via a coarse-graining map, which yields the diffusion limit without providing convergence rates. 
	
In \cite{EM10,Gou05,PoS00}, the qualitative analysis of diffusion limit from the Vlasov--Poisson--Fokker--Planck system to the drift-diffusion equation is investigated, with $K$ representing either attractive or repulsive Coulomb potentials. In \cite{PoS00}, strong convergence of $\rho_{f^\e}$ is established on short time intervals for $d=2,3$, and weak convergence is shown on global time intervals for $d=2$. This work was later extended in \cite{Gou05} for Newtonian interactions when $d=2$. These works were further improved in \cite{EM10}, where global convergence was achieved without restriction on the time interval or dimension. Later, it was extended to the kinetic model for multi-species charged particles in \cite{WLL15}. 

The approaches in these works primarily rely on uniform-in-$\e$ entropy estimates combined with compactness arguments. Our present work has some similarities to \cite{EM10} in that we do not impose additional strong assumptions on the initial data; instead, we only consider a weak solution that satisfies the kinetic entropy inequality. A quantitative convergence result was established in \cite{CT22} by utilizing the second-order Wasserstein distance, where the gradient-flow structure of the limit equation plays a crucial role. Notably, this work covers both attractive and repulsive Coulombic interactions with zero background states. In the case of nonzero background states, quantitative convergence results were obtained in \cite{Blapre, HR18} using a non-perturbative framework and in \cite{Zho22} using a perturbative approach both within the $L^2$ framework.

For the nonlinear diffusion case $(\gamma > 1)$, we focus on two recent works \cite{AT19, DMOS07}, where the BGK-type operator is considered. These studies are closely related to our work and are highlighted for clarity, given the extensive research conducted from various perspectives on this topic. In \cite{DMOS07}, the same local equilibrium $M_\gamma[\rho]$ is considered as a special case, and the equation \eqref{main_eq} without the nonlocal interaction potential $K$ is quantitatively derived for a uniformly bounded initial data. In fact, the general form of the local equilibrium is designed in \cite{DMOS07}, which is of the form:
\[
M_\gamma[\rho] = \zeta\lt(\frac{|v|^2}2 - \bar\mu(\rho) \rt),
\]
where $\zeta: \R \to \R_+$ and $\bar\mu$ is implicitly determined by 
\[
\rho = \intr \zeta\lt(\frac{|v|^2}2 - \bar\mu(\rho) \rt)dv.
\]
Under various assumptions on $\zeta$ and $V$, the following form of the drift-diffusion equation is rigorously derived:
\[
\pa_t \rho - \nabla_x \cdot (\rho \nabla_x V) = \Delta_x \nu(\rho),
\]
with 
\[
\nu(\rho)\coloneqq  \int_0^\rho \tilde \rho \bar\mu'(\tilde\rho)\,d\tilde\rho.
\]
A crucial idea in the proofs is the use of the maximum principle. Specifically, in the absence of a nonlocal interaction forcing term $-\nabla K * \rho$, the left-hand side of the kinetic equation \eqref{main_kin} becomes a linear transport equation. By carefully handling the BGK-operator, \cite{DMOS07} shows that $f^\e$ and $\rho_{f^\e}$ can be bounded uniformly in $\e$, i.e., $f^\e(t,\cdot,\cdot) \leq f_*$ and $\rho_{f^\e}(t, \cdot) \leq \rho_*$ for some uniform constants $f_*$ and $\rho_*$ for all $t > 0$ under an assumption on $f^\e_0$ related to $\zeta$. These bounds not only facilitate obtaining quantitative estimates on the local moment $m_{f^\e} = \intr vf^\e\,dv$ and other error terms but are also essential in applying compactness arguments effectively. However, this approach cannot be directly applied in the presence of a nonlocal interaction potential $K$ since the left-hand side of \eqref{main_kin} is no longer a linear transport equation. 

In a more recent work \cite{AT19}, the diffusion limit of the BGK--Poisson equation with $d\leq 3$ is studied, using a similar local equilibrium function as in \cite{DMOS07}, although under different assumptions on the function $\zeta$. As previously mentioned, in this case, the maximum principle cannot be directly applied. Instead, a {\it weak} maximum principle is established:
\[
	f^\e \leq f_*^\e = \zeta \lt(\frac{|v|^2}2 + K^\e * \rho_{f^\e} - \int_0^t \| K^\e * \pa_t\rho_{f^\e}\|_{L^\infty}\,ds \rt).
\]
To control the right-hand side of the above inequality uniformly in $\e>0$, $\zeta$ is assumed to be bounded. By utilizing this uniform bound on $f^\e$ in $L^\infty_{loc}(\R_+; L^1 \cap L^\infty(\bbR^d \times \bbR^d))$, along with useful $L^2$ estimates for weak solutions of the kinetic equation and the velocity averaging lemma, the convergence of weak solutions to the BGK--Poisson equation to a weak solution of the nonlinear drift-diffusion-Poisson equation is established. Notably, in the one-dimensional case, a quantitative bound on the convergence of $f^\e$ is obtained in $L^\infty([0,T]; L^1(\bbR^d \times \bbR^d))$. Here, we would like to point out that the local equilibrium function $M_\gamma[\rho]$ \eqref{pre} does not satisfy the assumptions stated in \cite{AT19}.

The primary motivation for this work is to extend the results of \cite{AT19, DMOS07} to the cases where the initial data satisfy only the natural bounds associated with the kinetic entropy inequality. More precisely, we focus on studying the convergence of {\it weak entropy solutions} to the equation \eqref{main_kin} toward the nonlinear aggregation-diffusion equation \eqref{main_eq}. For a precise definition of weak entropy solutions, see Definition \ref{def_weak} below. It is important to note that, in general, weak entropy solutions are not necessarily bounded. As a result, we cannot rely on the maximum principle and must develop new techniques to analyze the diffusion limit of \eqref{main_kin}. In this context, we only have a uniform $L^\infty_{loc}(\R_+;L^1 \cap L^{1+2/n}(\bbR^d \times \bbR^d))$ bound on $f^\e$ and an $L^\infty_{loc}(\R_+;L^\gamma(\bbR^d))$ bound on $\rho_{f^\e}$, which are derived from the natural kinetic entropy inequality. These limited bounds present significant challenges in achieving the necessary strong compactness for the solutions.  

\subsubsection{Other derivations of \eqref{main_eq}}
The nonlinear aggregation-diffusion equation \eqref{main_eq} can be rigorously derived from compressible Euler-type equations through various approaches involving large friction limits. The seminal work by \cite{C07} established the strong relaxation limit of multidimensional isothermal Euler equations. Subsequent studies have further explored this connection. For instance, \cite{CCT19} analyzed the convergence to equilibrium in Wasserstein distance for damped Euler equations with interaction forces, providing insights into the behavior of solutions under large friction. In a related context, \cite{GLT17} developed the theory of relative energy for the Korteweg model and other Hamiltonian flows in gas dynamics, contributing to a deeper understanding of the asymptotic behavior of such systems. The work by \cite{LT17} focused on the transition from gas dynamics with large friction to gradient flows, describing diffusion theories and further elucidating the connection between Euler-type systems and aggregation-diffusion equations. More recent contributions, such as \cite{ACC24, CJ21}, have investigated relaxation to fractional porous medium equations from the Euler--Riesz system, expanding the theory to include fractional diffusion effects.

Nonlinear diffusion equations of porous medium type have also been rigorously derived as macroscopic limits of interacting particle systems in various settings. The microscopic dynamics underlying these derivations typically fall into three broad classes: (1) on-lattice dynamics, such as zero-range processes, where the particle motion is constrained to a discrete lattice structure \cite{FIS97, GLT09, Inoue91, SU93}; (2) off-lattice dynamics, involving weakly interacting diffusions \cite{Dembo16, Sh12}; and
(3) singular limits of aggregation-type models, where particle interactions exhibit strong attraction or repulsion and the continuum limit leads to degenerate diffusion \cite{BE23, FP08, Oel90}. Depending on the context, the derivation of the porous medium equation from these systems can be carried out via a range of analytical frameworks, including hydrodynamic limits, mean-field limits, and large deviation principles.

		%%%%%%%%%%%%%%%%%%%%%%%%%%%%%%%%%%%%%%%%%%%%%%
	%
	%
	%
	%
	%
	%
	%%%%%%%%%%%%%%%%%%%%%%%%%%%%%%%%%%%%%%%%%%%%%%
	%\subsection{Novelty and strategy of the current work}

		%%%%%%%%%%%%%%%%%%%%%%%%%%%%%%%%%%%%%%%%%%%%%%
	%
	%
	%
	%
	%
	%
	%%%%%%%%%%%%%%%%%%%%%%%%%%%%%%%%%%%%%%%%%%%%%%
	\subsection{Main results}
	In this subsection, we precisely state our main results concerning the diffusion limit of the BGK-type equation \eqref{main_kin} toward the nonlinear aggregation-diffusion equation \eqref{main_eq}. For this, we first present the result on the global-in-time weak solutions to \eqref{main_kin} and subsequently analyze their behavior in the diffusion limit. We highlight that our results not only establish the convergence of solutions but also directly demonstrate the existence of weak solutions to the limit equation \eqref{main_eq}. 
	
Before proceeding, we outline the main assumptions on the potential functions $V$ and $K$.

\medskip
\paragraph{\bf (HV)} There exists a positive constant $C_V>0$ such that
\[
V \in W^{1,\infty}_{loc}(\bbR^d), \quad V(x) \geq -C_V, \quad \mbox{and} \quad |\nabla V(x) \cdot x| \leq C_V ( 1 + |V(x)|) \mbox{ for all $x \in \bbR^d$}.
\]

\medskip
\paragraph{\bf (HK)} There exists a positive constant $C_K>0$ such that
\[%\bq\label{condi_hk1}
	K(x) \geq -C_K, \quad |\nabla K(x) \cdot x| \leq  C_K(1 + |K(x)|) \quad \mbox{for } x \in \bbR^d
\]%\eq
and
\[%	\bq\label{condi_K}
K \in L^p(B_{2R}), \quad \nabla K \in L^q(B_{2R}), \quad \mbox{and} \quad \nabla K \in L^r(\bbR^d \backslash B_R) 
\]%\eq
for some constant $R>0$ and extended real numbers $p,q,r \in [1,\infty]$ satisfying 
\bq\label{condi_hk2}
0 \leq \frac1p < 2 - \frac2\gamma, \quad 0 \leq \frac1r < \min \lt\{1 - \frac1{1 + \frac 2n}, \ 2 - \frac 2\gamma \rt\}, \quad \mbox{and} \quad \frac1\gamma - 1 \leq \frac1r - \frac1q \leq 0.
\eq
Here we denoted $B_R := \{x \in \R^d : |x| < R\}$ for $R>0$.

\begin{remark} Note that
\[
 \min \lt\{1 - \frac1{1 + \frac 2n}, \ 2 - \frac 2\gamma \rt\} = \left\{ \begin{array}{ll}
1 - \frac1{1 + \frac 2n} & \textrm{if $\gamma \geq \frac d{d-1}$}\\[2mm]
2 - \frac 2\gamma & \textrm{otherwise}
  \end{array} \right..
\]
\end{remark}
		
\begin{remark}
A typical example of the potential $V$ is the quadratic confinement potential, i.e., $V(x) = |x|^2$, which clearly satisfies the assumptions stated above. The assumptions on $K$ allow us to consider a variety of interaction potentials. Below, we list some key examples. 

\begin{enumerate}[label=(\roman*)]
    \item Lipschitz continuous potentials:
    \begin{itemize}
        \item[--]  Morse potential: $K(x) = A e^{-|x|}$, where $A$ is a constant. This potential decays exponentially with distance.
        \item[--] Gaussian potential: $K(x) = A e^{-|x|^2}$, where $A$ is a constant. This potential has a bell-shaped curve and decays more quickly than the Morse potential.
    \end{itemize}

    \item Potentials with non-Lipschitz continuity near the origin:
    These potentials may not be Lipschitz continuous at the origin but still satisfy the required conditions:
    \[
    K(x) = C_a|x|^a e^{-\bar{C}_a |x|} + C_b|x|^b e^{-\bar{C}_b |x|},
    \]
    where $C_a, C_b \in \R$, $\bar{C}_a, \bar{C}_b > 0$ and $a, b \in (1 - d(\gamma-1)/\gamma, 1]$.
    This form represents a combination of terms that involve powers of $|x|$ and exponential decay. The parameters $a$ and $b$ control the behavior of the potential near the origin.

    \item Unbounded potentials:
    Some interaction potentials grow without bound as $|x|$ increases:
    \[
    K(x) = c_a |x|^a + c_b,
    \]
    where $c_a > 0$, $c_b$ is a constant, and $a \in (1 - d(\gamma-1)/\gamma, 1]$.
    This potential grows as a power of $|x|$ and can be used to model interactions that become stronger with increasing distance.
    \item 3D Coulomb potential:
Consider the Coulomb potential 
    \[
    	K(x) = |x|^{-1}
    \]
    in three dimensions with $\gamma = 7/5$. This potential describes long-range repulsive forces, such as electrostatic interactions in three dimensions. In this case, the condition \eqref{condi_hk2} becomes
    \[
0 \leq \frac1p < \frac47, \quad    0 \leq \frac 1r < \frac12, \quad \mbox{and} \quad \frac1q \leq \frac1r + \frac27 \leq \frac1q + \frac27.
    \]
In particular, choosing $p,q,r$ such that
    \[
    \frac13 < \frac1p < \frac47, \quad \frac23 < \frac 1q < \frac57, \quad \mbox{and} \quad \frac37 < \frac1r < \frac12,
    \]
ensures that the required integrability conditions hold. This region is depicted as a dark region in the right figure of Figure \ref{fig:exponents}.
\end{enumerate}
\end{remark}

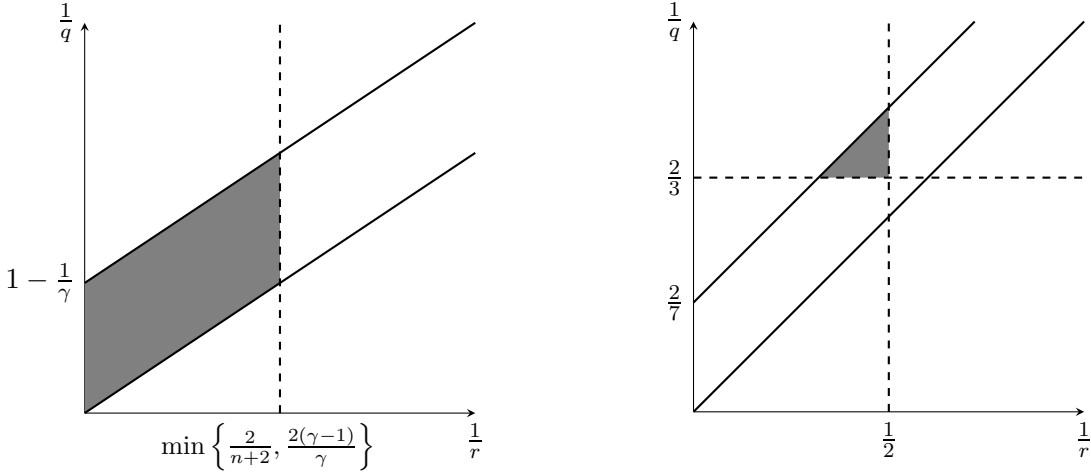
\begin{figure}[h]
\centering
\begin{minipage}{0.48\textwidth}
\centering
\vspace{0.2cm}
\begin{tikzpicture}
    \begin{axis}[
        scale=0.7,
        width=9cm, height=9cm,
     xmin=0, ymin=0,
        axis x line=middle,
        axis y line=middle,,
        xlabel={$ \frac{1}{r}$},
        ylabel={$\frac{1}{q}$},
        xtick={0},
        ytick={0},
        xticklabels={0},
        yticklabels={0},
        ticklabel style={font=\small},
        xlabel style={at={(axis description cs:1,0)},anchor=north},
        ylabel style={at={(axis description cs:0,1)},anchor=east},
        samples=200,
        domain=0:1,
        clip=false
    ]
        % Lower line
        \addplot[name path=line1, thick] {x}; % Line y = x

        % Upper line slightly farther away
        \addplot[name path=line2, thick, domain=0:1] {x + 0.5}; % Offset adjusted to 0.15

        % Vertical dashed line
        \draw[dashed, thick] (axis cs:0.5,0) -- (axis cs:0.5,1.5);

        % Fill the gray region between the lines
        \addplot[gray] fill between[of=line1 and line2, soft clip={domain=0:0.5}];

        % Add labels
        \node[below] at (axis cs:0.47,0) {\small{$ \min\left\{\frac2{n+2}, \frac{2(\gamma - 1)}{\gamma}\right\}$}};
        \node[left] at (axis cs:0,0.5) {$ 1 - \frac{1}{\gamma}$}; % Adjusted to align with the upper line
    \end{axis}
\end{tikzpicture}
\end{minipage}
\quad
 \begin{minipage}{0.48\textwidth}
 \centering
\begin{tikzpicture}
    \begin{axis}[
    scale=0.7,
        width=9cm,
        height=9cm,
        xmin=0, xmax=1,
        ymin=0, ymax=1,
        axis x line=middle,
        axis y line=middle,
        xlabel={$\frac{1}{r}$},
        ylabel={$\frac{1}{q}$},
        xtick={0},
        ytick={0},
        xticklabels={$0$, $\frac{1}{2}$, $1$},
        yticklabels={$0$, $\frac{2}{7}$, $\frac{2}{3}$, $\frac{1}{q}$},
        ticklabel style={font=\small},
        xlabel style={at={(axis description cs:1,0)},anchor=north},
        ylabel style={at={(axis description cs:0,1)},anchor=east},
        domain=0:1,
        clip=false
    ]
    
%        \path[name path=line0] (axis cs:0.5,0.66) -- (axis cs:0.5,0.78);
        \path[name path=line0] (axis cs:0.32,0.6) -- (axis cs:0.5,0.6);
        
        % Lower diagonal line (y = x)
        \addplot[name path=line1, thick] {x};
        
        % Upper diagonal line (y = x + 2/7)
        \addplot[name path=line2, thick, domain=0:0.72] {x + 0.28};
        
        % Vertical dashed line at x = 0.5
        \draw[dashed, thick] (axis cs:0.5,0) -- (axis cs:0.5,1);
        
        % Horizontal dashed line at y = 2/3
        \draw[dashed, thick] (axis cs:0,0.6) -- (axis cs:1,0.6);
        
        % Fill the black triangular region
       \addplot[gray] fill between[of=line2 and line0, soft clip={domain=0.32:0.5}];

% \addplot[black] fill between[of=line2 and axis_y, soft clip={domain=0.5:2/3}];
        
        % Labels for specific points
        \node[below] at (axis cs:0.5,0) {$\frac{1}{2}$};
        \node[left] at (axis cs:0,0.6) {$\frac{2}{3}$};
        \node[left] at (axis cs:0,0.28) {$\frac{2}{7}$};
    \end{axis}
\end{tikzpicture}
\end{minipage}

	\caption{Conditions for the exponents $q$ and $r$. The left figure indicates the region of $(\frac{1}{r},\frac{1}{q})$ that satisfying the condition \eqref{condi_hk2}, while the right figure shows the region of exponents for the 3D Coulomb potential case with $\gamma=7/5$.}
	\label{fig:exponents}
\end{figure}

Without loss of generality, we will assume in the following that the potentials $V$ and $K$ satisfy $V, K \geq 0$; if necessary, we may consider $(V, K) \mapsto (V + C_V, K + C_K)$, and thus for $U \in \{ V, K\}$,
\[
	U(x) \geq 0 \quad \mbox{and} \quad |\nabla U(x) \cdot x| \leq C_U ( 1 + U(x)) \quad \mbox{for all $x \in \bbR^d$}.
\]

		%%%%%%%%%%%%%%%%%%%%%%%%%%%%%%%%%%%%%%%%%%%%%%
	%
	%
	%
	%
	%
	%
	%%%%%%%%%%%%%%%%%%%%%%%%%%%%%%%%%%%%%%%%%%%%%%

	\subsubsection{\bf\em Global existence of weak entropy solutions}

	\begin{definition}\label{def_weak} For a given $T>0$, we say that $f$ is a \emph{weak entropy solution} to \eqref{main_kin} subject to the initial datum $f_0$ if the following conditions are satisfied: 
		\begin{itemize}
			\item[(i)] $f \in L^\infty([0,T]; L^1_2 \cap L^{1+2/n}(\bbR^d \times \bbR^d))$, where
			\[
				L^1_2(\bbR^m) = \left\{ f\in L^1(\bbR^m) : \int_{\bbR^m} |z|^2f(z)\,dz<+\infty\right\},\quad m\in\bbN,
			\]
			denotes the space of integrable functions with finite second moment.
%			where $r > 0$ is given by 
%			\bq\label{const_r}
%			r = \left\{ \begin{array}{ll}
%				\displaystyle 1 + \frac 2n &   \displaystyle \textrm{if $\gamma \in \lt(1, 1 + \frac2d\rt)$}\\[3mm]
%				\displaystyle \infty &   \displaystyle \textrm{if $\gamma = 1 + \frac2d$}
%			\end{array} \right..
%			\eq
			\item[(ii)] For all $\varphi \in \mc^1_c([0,T]\times \bbR^d \times \bbR^d)$ with $\varphi(T,z) = 0$, $z=(x,v),$
			\begin{align*}
				- \intrr \varphi(0,\cdot)\,f_0\,dz - \int_0^T \!\!\! \intrr \bigl(\e \pa_t \varphi + v \cdot \nabla_x \varphi &- \sfF[\rho_f] \cdot \nabla_v \varphi \bigr) f dz\,dt \cr
				& = \frac{1}{\e}\int_0^T\!\!\! \intrr \varphi\bigl(M_\gamma[\rho_f] - f\bigr) dz\,dt,
			\end{align*}
			where $\sfF[\rho] = \nabla V + \nabla K\ast\rho$ is the (nonlocal) force.
			\item[(iii)] The \emph{energy inequality} holds 		for almost every $t \in [0,T]$:
		\begin{align}\label{kin_ineq}
			\calE[f(t,\cdot)] + \frac{1}{\e^2}\int_0^t\intrr \bigl(H[f](s,z) - H[M_\gamma[\rho_f]](s,z) \bigr)\,dz \,ds  \le \calE[f_0], 
		\end{align}
		where
		\[
			\calE[f] \coloneqq \intrr H[f](z)\,dz + \intr V\rho_f\,dx + \frac12 \intr \rho_f K\ast\rho_f\,dx.
		\]
		Here, the \emph{kinetic entropy} is defined for $n\ge 2$ as 
				\[%\bq\label{kin_ent}
		H[f](z) = 
			\displaystyle \frac{|v|^2}2f(z) + \frac1{2c_{\gamma,d}^{2/n}} \frac{f^{1 + 2/n}(z)}{1 + 2/n}, \qquad z=(x,v)\in\bbR^d\times\bbR^d.
		\]%\eq
		\end{itemize}
	\end{definition}

	\begin{theorem}\label{thm_kext} Let $T>0$. Further, let $V$ and $K$ satisfy {\bf (HV)} and {\bf (HK)}, respectively. Then, for any initial datum $f_0\in L^1_2 \cap L^{1+2/n}(\bbR^d \times \bbR^d)$ with finite energy $\calE[f_0]<+\infty$, there exists a weak entropy solution $f$ to the equation \eqref{main_kin} in the sense of Definition \ref{def_weak}.
	\end{theorem} 		
		
 			%%%%%%%%%%%%%%%%%%%%%%%%%%%%%%%%%%%%%%%%%%%%%%
	%
	%
	%
	%
	%
	%
	%%%%%%%%%%%%%%%%%%%%%%%%%%%%%%%%%%%%%%%%%%%%%%

\subsubsection{\bf\em Diffusion limit of \eqref{main_kin}: Derivation of \eqref{main_eq}}

\begin{theorem}\label{thm_main}
Let the assumptions of Theorem \ref{thm_kext} hold, and $\{f^\e\}_{\e>0}$ be a family of weak entropy solutions to \eqref{main_kin} with initial data $\{f_0^\e\}_{\e>0}$ satisfying $\sup_{\e>0}\calE[f_0^\e] <+\infty$. Suppose that
\[
\rho_{f^\e_0} \rightharpoonup \rho_0 \quad \mbox{weakly in $L^\gamma(\bbR^d)$ as $\e\to 0$}.
\]
Then, there exists $\rho \in L^\gamma((0,T) \times \bbR^d)$ with $\sup_{t\in[0,T]} \calF[\rho(t,\cdot)] <+\infty$, where
\[
	\calF[\rho] = \left(1+\frac{d}{2}\right)\|\rho\|_{L^\gamma}^\gamma + \intr V\rho\,dx + \frac12 \intr \rho K\ast\rho\,dx,
\]
such that
\begin{gather*}
  \rho_{f^\e} \to \rho \quad \text{in } L^p((0,T) \times \bbR^d) \quad \text{as $\e\to 0$\; for any } p \in [1,\gamma), \\
  		\rho^\e(t,\cdot)\,\calL^d \rightharpoonup \rho(t,\cdot)\,\calL^d\quad\text{narrowly in $\calP(\bbR^d)$}\quad\text{for every $t\in[0,T]$,}\\
  f^\e \to M_\gamma[\rho] \quad \text{in } L^1((0,T); L^1_{loc}(\bbR^d \times \bbR^d))\quad\text{as $\e \to 0$,}
\end{gather*}
 and $\rho$ is a weak solution to the equation \eqref{main_eq} with the initial condition $\rho(0,x) = \rho_0 \in L^\gamma(\R^d)$, i.e.,
 \[
\intr \varphi(0,x)\rho_0(x)\,dx + 	\int_0^T\!\!\! \intr \partial_t\varphi\,\rho\,dxdt = \int_0^T\!\!\! \intr \big(\Delta\varphi\,\rho^\gamma + \nabla\varphi \cdot \rho \sfF[\rho]\big)\,dxdt,
	\]
	for every $\varphi\in \calC_c^\infty([0,T)\times \bbR^d)$. Here, $\calL^d$ is the $d$-dimensional Lebesgue measure on $\bbR^d$.
\end{theorem}

		%%%%%%%%%%%%%%%%%%%%%%%%%%%%%%%%%%%%%%%%%%%%%%
	%
	%
	%
	%
	%
	%
	%%%%%%%%%%%%%%%%%%%%%%%%%%%%%%%%%%%%%%%%%%%%%%
	%\subsection{Organization of paper}

The remainder of this paper is organized as follows. In Section~\ref{sec:mini_p}, we discuss relevant properties of the local equilibrium, including the fact that it satisfies a minimization principle (cf.\ Lemma~\ref{lem:minimization}). This minimization principle offers effective control of $L^{1+2/n}$ dissipation estimates for weak entropy solutions. Section~\ref{sec_uni} is dedicated to establishing the uniform bounds in $\e$ for the kinetic entropy and observable quantities $\rho_{f^\e}$ and $m_{f^\e}$. Under the assumptions on the potentials $V$ and $K$ as stated in Theorem~\ref{thm_main}, we demonstrate uniform bound estimates on $f^\e$, $\rho_{f^\e}$, and $m_{f^\e}$. Given that our analysis is conducted in an unbounded spatial domain, we emphasize the need for careful analysis to establish compactness properties of the solutions. We further explore the relationship between entropy dissipation and relative entropy. In Section~\ref{sec_str}, we employ the velocity averaging lemma to achieve the required strong compactness of $\rho_{f^\e}$. This step is crucial due to the nonlinear diffusion term and the nonlocal interaction term.  Although the velocity averaging lemma is well-developed in various contexts, we provide detailed proof relevant to our specific setting since no direct references are available. Utilizing the estimates obtained in the preceding sections, we prove Theorem~\ref{thm_main} in Section~\ref{sec_pf}. Finally, in Section~\ref{sec_exi}, we demonstrate the existence of weak entropy solutions to the kinetic equation \eqref{main_kin}, thereby proving Theorem~\ref{thm_kext}.

	%%%%%%%%%%%%%%%%%%%%%%%%%%%%%%%%%%%%%%%%%%%%%%
	%
	%
	%
	%
	%
	%
	%%%%%%%%%%%%%%%%%%%%%%%%%%%%%%%%%%%%%%%%%%%%%%
	\section{Properties of the local equilibrium}\label{sec:mini_p}
	This section presents preliminary results regarding the local equilibrium map that will be used in the subsequent analysis. We begin by showing that the local equilibrium minimizes the kinetic entropy $H[f]$. More precisely, let $m_\gamma:[0,\infty)\times \bbR^d \to[0,\infty)$ be defined for $\gamma \in (1, 1 +\frac{2}{d+2}]$ by
	\begin{align}\label{eq:def_m}
		m_\gamma(\rho,v) \coloneqq c_{\gamma,d}\bigl(b_0\rho^{\gamma-1}-|v|^2\bigr)^{n/2}_+,\qquad b_0 \coloneqq \frac{2\gamma}{\gamma-1} = 2\gamma'.
	\end{align}
	We further define the function $\Psi_n\colon [0,+\infty)\to [0,+\infty)$ given by
		\[
			\Psi_n(s)=
		\displaystyle\frac{1}{2c_{\gamma,d}^{2/n}}\frac{s^{1+2/n}}{1+2/n}, \quad n > 0		.
		\]
	Note that $\Psi_n$ is a differentiable, convex function, and that the kinetic entropy can be expressed as
	\[
		H[f](z) = \frac{|v|^2}{2} f(z) + \Psi_n(f(z)),\qquad z\in\bbR^d\times\bbR^d.
	\]

	Then, the following result holds.
	
	\begin{lemma}\label{lem:minimization}
		For every $\rho\in[0,+\infty)$, the function $\bbR^d\ni v\mapsto m_\gamma(\rho,v)$ minimizes the problem 
		\[
			\min\biggl\{ \intr H[g](v)\,dv ~\bigg|~ g\colon\bbR^d\to [0,+\infty)\;\;\text{measurable},\;\intr g(v)\,dv = \rho \biggr\}.
		\]
		Moreover, we have that
		\[
			\intr |v|^2\,m_\gamma(\rho,v)\,dv = d\rho^\gamma,\qquad \intr \Psi_n(m_\gamma(\rho,v))\,dv = \rho^\gamma.
		\]
	\end{lemma}
	\begin{proof}
		Due to the convexity of $\Psi_n$, we have for any $s_0, s_1\in[0,\infty)$ that
		\[
		\Psi_n(s_1)\ge \Psi_n(s_0) +\Psi_n'(s_0)(s_1-s_0),\qquad \Psi_n'(s) = \frac{s^{2/n}}{2c_{\gamma,d}^{2/n}}.
		\]
Substituting $s_0$ and $s_1$ with $m_\gamma(\rho,\cdot)$ and an arbitrary function $g=g(v)$, respectively, we obtain
		\begin{equation}\label{ineq}
		\Psi_n(g(v))\ge \Psi_n(m_\gamma(\rho,v))+\Psi'_n(m_\gamma(\rho,v))(g-m_\gamma(\rho,v)).
		\end{equation}
		On the other hand, since
		\[
			\Psi_n'(m_\gamma(\rho,v)) = \frac{1}{2c_{\gamma,d}^{2/n}}\left(m_\gamma(\rho,v)\right)^{2/n}=\frac{1}{2}\left(b_0\rho^{\gamma-1}-|v|^2\right)_+,
		\]
integrating \eqref{ineq} over $v\in\bbR^d$, we deduce
		\begin{align}
		\begin{aligned}\label{eq:psi_n_mini}
		\intr\Psi_n(g(v))\,dv&\ge \intr\Psi_n(m_\gamma(\rho,v))\,dv +\frac{1}{2}\intr\left(b_0\rho^{\gamma-1}-|v|^2\right)_+(g(v)-m_\gamma(\rho,v))\,dv\\
		&= \intr\Psi_n(m_\gamma(\rho,v))\,dv + \frac{1}{2}\intr\left(b_0\rho^{\gamma-1}-|v|^2\right)(g(v)-m_\gamma(\rho,v))\,dv\\
		&\qquad +\frac{1}{2}\intr \left(b_0\rho^{\gamma-1}-|v|^2\right)_-(g(v)-m_\gamma(\rho,v))\,dv.
		\end{aligned}
		\end{align}
		Now, observing
		\[
			\intr H[g](v)\,dv = \intr \left(\frac{|v|^2}{2}g(v)+\Psi_n(g(v))\right)dv,
		\]
		and rearranging the terms in \eqref{eq:psi_n_mini}, we get
		\begin{align*}
		\intr H[g](v)\,dv 
		&\ge \intr H[m_\gamma(\rho,\cdot)](v)\,dv
		+\frac{1}{2}\intr \left(b_0\rho^{\gamma-1}-|v|^2\right)_-(g(v)-m_\gamma(\rho,v))\,dv,
		\end{align*}
		where we used the fact that $\intr(g(v)-m_\gamma(\rho,v))\,dv=0$. Additionally, since $m_\gamma(\rho,\cdot)$ vanishes for every $|v|^2\ge b_0\rho^{\gamma-1}$, we have that
		\[
			\left(b_0\rho^{\gamma-1}-|v|^2\right)_-m_\gamma(\rho,v)=0\qquad\text{for all $v\in\bbR^d$}.
		\]
		This implies the desired minimization principle:
		\begin{align*}
		\intr H[g](v)\,dv &\ge \intr H[m_\gamma(\rho,\cdot)](v)\,dv,
		\end{align*}
		and this completes the proof for the first part of the statement.
		
		As for the second part, recall from the introduction that
		\[
		\intr v\otimes v\, m_\gamma(\rho,v)\,dv = \rho^\gamma\, \mathbb{I}_d.
		\]
		Hence, taking the trace yields
		\[
		\intr |v|^2\, m_\gamma(\rho,v)\,dv = d \rho^\gamma.
		\]
		As for the other term, straightforward computation yields
		\begin{align*}
			\intr \Psi_n(m_\gamma(\rho,v))\,dv &= \frac{1}{2c_{\gamma,d}^{2/n}}\frac{1}{1+2/n}\intr \Bigl(c_{\gamma,d}\bigl(b_0\rho^{\gamma-1}-|v|^2\bigr)^{n/2}_+\Bigr)^{1+2/n}dv \\
			&= \frac{c_{\gamma,d}}{2}\frac{1}{1+2/n} b_0^{\gamma/(\gamma-1)}\rho^\gamma\intr \bigl(1 - |\xi|^2\bigr)_+^{1+n/2} d\xi = \rho^\gamma,
		\end{align*}
		where we used the fact that
		\[
			\intr \bigl(1 - |\xi|^2\bigr)_+^{1+n/2} d\xi 
			= 2(1+n/2) b_0^{-\gamma/(\gamma-1)} c_{\gamma,d}^{-1}.
		\]
		Putting the terms together gives the required expression.
	\end{proof}
	
	From its definition, it is clear to see that $m_\gamma$ satisfies the usual Carath\'eodory conditions. Together with Lemma~\ref{lem:minimization} and the property
	\[
		\intr m_\gamma(\rho,v)\,dv = \rho,
	\]
	we conclude that the map
	\[
		M_\gamma[\rho](x,v) \coloneqq m_\gamma(\rho(x),v)
	\] 
	then defines a Nemytskii operator $M_\gamma\colon \frkR^\gamma\to \frkX^\gamma$ with the spaces $\frkR^\gamma \coloneqq L^1\cap L^\gamma(\bbR^d)$ and
	\[
		\frkX^\gamma \coloneqq \left\{ f\in L^1\cap L^{1+2/n}(\bbR^d\times\bbR^d): \intrr \bigl(1+|v|^2\bigr)\,f(x,v)\,dx\,dv < +\infty\right\}.
	\]
	Then, we will show that the Nemytskii operator $M_\gamma$ is locally Lipschitz below in Lemma \ref{lem:nemytskii}.

    %%%%%%%%%%%%%%%%%%%%%%%%%%%%%%%%%%%%%%%%%%%%%%
	%
	%
	%
	%
	%
	%
	%%%%%%%%%%%%%%%%%%%%%%%%%%%%%%%%%%%%%%%%%%%%%%	
 
	\subsubsection*{An extension map}
	
	Before we prove that the Nemytskii operator is locally Lipschitz, we introduce a useful extension. For a nonnegative function $g\colon\bbR^d\to \bbR_+$, we define its extension $\widehat{g}$ by (cf.\ \cite{KSpre})
	\begin{align}\label{eq:extension}
		\bbV\coloneqq \bbR^d\times \R_+ \ni (v,I)\mapsto \widehat g(v,I) \coloneqq b_2  \mathbf{1}_{[0,\,c_{\gamma, d}^{-1} g(v)]} \bigl(I^n\bigr)\;\;\in [0,b_2], 
	\end{align}
	where $c_{\gamma, d}$ and $b_0$ are given in \eqref{para} and \eqref{eq:def_m} respectively, and
	\[
	b_2 \coloneqq \bigl(\pi b_0\bigr)^{-\frac{1}{\gamma-1}}\Gamma\left(\frac{b_0}{2}\right).
	\]
	Note that the function $g$ is recovered by integrating the extension mapping $\widehat g$ with respect to the measure $\sigma(dI)\coloneqq b_1 I^{n-1} dI$:
	\[
	g(v) = \int_{\R_+} \widehat g(v,I)\, \sigma(dI), \qquad b_1 = \frac{2\pi^{n/2}}{\Gamma(n/2)},
	\]
	where we used the fact that $b_1 b_2 = n c_{\gamma, d}$.	Moreover, we find
	\[
	H[g] = \frac12 \int_{\R_+} \bigl(|v|^2 +I^2\bigr)\, \widehat g(v,I)\,\sigma(dI), \qquad \gamma \in \lt(1, 1 + \frac2{d+2} \rt].
	\]
	Notice that $\widehat m_\gamma$ can be expressed as
	\[
		\widehat m_\gamma(\rho,v,I) = b_2\mathbf{1}_{[0,\,b_0\rho^{\gamma-1}]}\bigl(|v|^2 + I^2\bigr).
	\]	
	This together with Lemma \ref{lem:minimization} yields
	\[
		\widehat m_\gamma(\rho,\cdot) = \argmin_{\widehat g}\biggl\{ \iint_\bbV  \bigl(|v|^2 +I^2\bigr)\,\widehat g(v,I) \,\sigma(dI)\,dv : \iint_\bbV\widehat g(v,I) \,\sigma(dI)\,dv = \rho \biggr\}.
	\]
In particular, we have
	\bq\label{diss_eq}
	0\le \intr \lt(H[f] - H[M_\gamma[\rho]]\rt) dv = \frac12\iint_\bbV \bigl(|v|^2 +I^2\bigr)\bigl(\widehat f - \widehat M_\gamma[\rho]\bigr)\sigma(dI)\,dv
	\eq
	for every $f$ with $\intr f\,dv = \rho$ and where $\widehat f$ denotes the extension of $f$.
 
 	\medskip
 
	The extension map above will not only simplify some proofs but will assist in obtaining uniform-in-$\e$ moment control for the BGK-type operator $Q_\gamma[f]$ (cf.\ Proposition~\ref{prop51}). Now, using the extension, we prove that the Nemytskii operator is locally Lipschitz.

	\begin{lemma}\label{lem:nemytskii}
		The Nemytskii operator $M_\gamma\colon \frkR^\gamma\to \frkX^\gamma$ is locally Lipschitz. More precisely,
		\begin{align*}
			\begin{aligned}
			\|M_\gamma[\rho]-M_\gamma[\eta]\|_{L^1} &\le a_\gamma\|\rho-\eta\|_{L^1},\\
			\intrr |v|^2\,|M_\gamma[\rho]-M_\gamma[\eta]|\,dz &\le b_\gamma\bigl(\|\rho\|_{L^\gamma}^{\gamma-1} + \|\eta\|_{L^\gamma}^{\gamma-1}\bigr)\|\rho-\eta\|_{L^\gamma}, \\
			\|M_\gamma[\rho]-M_\gamma[\eta]\|_{L^{1+2/n}} &\le b_\gamma c_{\gamma,d}^{2/n} \bigl(\|\rho\|_{L^\gamma}^{\gamma-1} + \|\eta\|_{L^\gamma}^{\gamma-1}\bigr)\|\rho-\eta\|_{L^\gamma},
			\end{aligned}
			\qquad \forall\,\rho,\eta\in\frkR^\gamma,
		\end{align*}
		with constants $a_\gamma = 2n$ and $b_\gamma = 2n b_0$.
	\end{lemma}
	\begin{proof}
		Let $\rho,\eta\in[0,+\infty)$ with $\rho\ge \eta$. Then,
		\begin{align*}
			\intr \bigl(1 + |v|^2\bigr)|m_\gamma(\rho,v)-m_\gamma(\eta,v)|\,dv &\le \iint_\bbV \bigl(1 + |v|^2\bigr)|\widehat m_\gamma(\rho,v,I)- \widehat m_\gamma(\eta,v,I)|\,\sigma(dI)\,dv \\
			&\hspace{-6em}= b_2 \iint_{\{b_0\eta^{\gamma-1}\le I^2 + |v|^2 \le b_0 \rho^{\gamma-1}\}} \bigl(1 + |v|^2\bigr) \,\sigma(dI)\,dv \\
			&\hspace{-6em}\le b_2\Bigl(1 + b_0\rho^{\gamma-1}\bigr)\iint_{\{b_0\eta^{\gamma-1}\le I^2 + |v|^2 \le b_0 \rho^{\gamma-1}\}} \sigma(dI)\,dv \\
			&\hspace{-6em}= b_1b_2\Bigl(b_0^{1/(\gamma-1)} + b_0^{\gamma/(\gamma-1)}\rho^{\gamma-1}\bigr)|\rho-\eta|\intr \bigl(1-|v|^2\bigr)_+^{n/2}dv \\
			&\hspace{-6em}= 2n |\rho-\eta| + 2n b_0\rho^{\gamma-1}|\rho-\eta| \eqqcolon \bigl(\alpha_\gamma + b_\gamma \rho^{\gamma-1}\bigr)|\rho-\eta|.
		\end{align*}
		Furthermore, we have that
		\begin{align*}
			\intr |m_\gamma(\rho,v)-m_\gamma(\eta,v)|^{1+2/n}\,dv &\le \intr m_\gamma^{2/n}(\rho,v)|m_\gamma(\rho,v)-m_\gamma(\eta,v)|\,dv \\
			&\le c_{\gamma,d}^{2/n}b_0\rho^{\gamma-1}\intr |m_\gamma(\rho,v)-m_\gamma(\eta,v)|\,dv \\
			&\le b_\gamma c_{\gamma,d}^{2/n}\rho^{\gamma-1}|\rho-\eta|.
		\end{align*}
		Now let $\rho,\eta\in \frkR^\gamma$, then the first estimate above can be used to deduce
		\begin{align*}
			\bigl\|\bigl(1+|v|^2\bigr)\bigl(M_\gamma[\rho]-M_\gamma[\eta]\bigr)\bigr\|_{L^1} &= \int_{\{\rho\ge \eta\}} \int_\bbV \bigl(1 + |v|^2\bigr)|m_\gamma(\rho(x),v)-m_\gamma(\eta(x),v)|\,\sigma(dI)dv dx \\
			&\qquad+\int_{\{\rho < \eta\}} \int_\bbV \bigl(1 + |v|^2\bigr)|m_\gamma(\rho(x),v)-m_\gamma(\eta(x),v)|\,\sigma(dI)dvdx \\
			&\hspace{-10em}\le a_\gamma \|\rho-\eta\|_{L^1} + b_\gamma\biggl[ \int_{\{\rho\ge \eta\}} \rho^{\gamma-1}(x)|\rho(x)-\eta(x)|\,dx + \int_{\{\rho < \eta\}} \eta^{\gamma-1}(x)|\rho(x)-\eta(x)|\,dx\biggr] \\
			&\hspace{-10em}\le a_\gamma \|\rho-\eta\|_{L^1} + b_\gamma \bigl(\|\rho\|_{L^\gamma}^{\gamma-1} + \|\eta\|_{L^\gamma}^{\gamma-1}\bigr)\|\rho-\eta\|_{L^\gamma},
		\end{align*}
		while the second estimate yields
		\begin{align*}
			\|M_\gamma[\rho]-M_\gamma[\eta]\|_{L^{1+2/n}} &\le b_\gamma c_{\gamma,d}^{2/n}\bigl(\|\rho\|_{L^\gamma}^{\gamma-1} + \|\eta\|_{L^\gamma}^{\gamma-1}\bigr)\|\rho-\eta\|_{L^\gamma},
		\end{align*}
		which, together, implies the local Lipschitz property of the map $M_\gamma$.
	\end{proof}

	%%%%%%%%%%%%%%%%%%%%%%%%%%%%%%%%%%%%%%%%%%%%%%
	%
	%
	%
	%
	%
	%
	%%%%%%%%%%%%%%%%%%%%%%%%%%%%%%%%%%%%%%%%%%%%%%
	\section{Uniform-in-$\e$ estimates}\label{sec_uni}
	
	 In this section, we provide a priori energy inequality estimates and their implications. For notational convenience, we denote $\rho^\e \coloneqq \rho_{f^\e}$, $m^\e \coloneqq m_{f^\e}$, and $u^\e \coloneqq u_{f^\e}$.	
	\begin{lemma}\label{lem_energy} Let $f^\e$ be a solution to \eqref{main_kin} with sufficient regularity. Then we have
\begin{align*}
		\calE[f^\e(t,\cdot)] + \frac{1}{\e^2}\int_0^t\!\!\intrr \bigl(H[f^\e] - H[M_\gamma[\rho^\e]] \bigr)\,dz\,ds   \le \calE[f_0^\e], \qquad\forall\,t\in[0,T].
\end{align*}
	\end{lemma}
	\begin{proof}
        Direct computation gives
		\[
		\frac{d}{dt}\lt(\frac12\intrr |v|^2 f^\e dz + \intr V \rho^\e dx + \frac12 \intr \rho^\e K * \rho^\e dx \rt) = \frac{1}{2\e^2} \intrr |v|^2(M_\gamma[\rho^\e] - f^\e)\,dz. 
		\]
		To provide a bound for the kinetic entropy, let $p=1+2/n$. Using Young's inequality, we derive
		$$\begin{aligned}
			\frac{d}{dt}\intrr (f^\e)^p\,dz 
%			= \frac{p}{\e^2}\iint (f^\e)^{p-1} Q_\gamma[f^\e]\,dz\\
			&=-\frac{p}{\e^2}\intrr (f^\e)^{p}\,dz + \frac{p}{\e^2}\intrr (f^\e)^{p-1}M_\gamma[\rho^\e]\,dz\\
			&\le -\frac{p}{\e^2}\intrr (f^\e)^{p}\,dz + \frac{p}{\e^2}\intrr \left(\frac{p-1}{p}(f^\e)^{p}+\frac{1}{p}\bigl(M_\gamma[\rho^\e]\bigr)^p\right) dz\\
			&=-\frac1{\e^2}\intrr (f^\e)^{p}\,dz + \frac1{\e^2}\intrr \bigl(M_\gamma[\rho^\e]\bigr)^p\,dz.
		\end{aligned}$$
		Thus, we have
\begin{align*}
		\frac{d}{dt}\calE[f^\e(t,\cdot)] \leq \frac1{\e^2} \intrr \bigl(H[M_\gamma[\rho^\e]](t,z) - H[f^\e](t,z) \bigr)\, dz.
\end{align*}
	This completes the proof after integrating in time. We remark that Lemma~\ref{lem:minimization} implies that the dissipation term is nonnegative.
		\end{proof}

		Based on the uniform-in-$\e$ estimate on the kinetic entropy, we now provide the uniform-in-$\e$ bound estimates on $\rho^\e$ and $m^\e$ in the proposition below.
		
		\begin{proposition}\label{prop_uni} Let $f^\e$ be a weak entropy solution to \eqref{main_kin}. Suppose that the potentials $V$ and $K$ satisfy the assumptions {\bf (HV)} and {\bf (HK)}, respectively. Moreover, we assume that  the initial total energy is bounded uniformly in $\e$, i.e.,
		\[
		\sup\nolimits_{\e >0}  \calE[f_0^\e] < \infty.
		\]
		Then, we have $f^\e \in L^\infty([0,T]; \frkX^\gamma)$,
		\[
		\rho^\e \in L^\infty([0,T]; \frkR^\gamma), \quad \mbox{and} \quad m^\e \in L^\infty([0,T]; L^p(\bbR^d)) \qquad \text{for all\, $p \in \lt[1,\frac{2\gamma}{\gamma+1}\rt]$}
		\]
		uniformly in $\e>0$. 
		\end{proposition}
		\begin{proof} 
			%For notational simplicity, we drop the domain of integration. 
            From Lemma \ref{lem_energy}, we obtain
			\begin{align*}
				\begin{aligned}%\label{ineq_ent}
					\sup_{t\in[0,T]}\intrr H[f^\e](t,z)\,dz  &\leq \sup_{t\in[0,T]}\calE[f^\e(t,\cdot)] \le \calE[f_0^\e] < +\infty
				\end{aligned}
			\end{align*}
			uniformly in $\e\in(0,1)$. In particular, $f^\e \in L^\infty([0,T]; \frkX^\gamma)$ is uniformly bounded.
			
			We now show that the previous uniform-in-$\e$ estimate implies the desired estimates for $\rho^\e$ and $m^\e$. Indeed, from Lemma \ref{lem:minimization}, we obtain
			\bq\label{pres}
				\intr (\rho^\e)^\gamma\,dx = \intrr H[M_{\gamma}[\rho^\e]]\,dz \leq \intrr H[f^\e]\,dz. 
			\eq
			Thus, $\|\rho^\e\|_{L^\gamma}$ is uniformly bounded. The uniform bound for $\|\rho^\e\|_{L^1}$ follows from the conservation of mass property for solutions $f^\e$ with sufficient regularity.
			
			Next, by Jensen's inequality, we get
			\[%\begin{equation}\label{eq:jensen}
			\intr \rho^\e |u^\e|^2\,dx  \leq \intrr |v|^2 f^\e \,dz,\qquad u^\e \coloneqq \frac{m^\e}{\rho^\e}.
			\]%\end{equation}
			Combining these results, we get the following bound for $m^\e$ for any $p \in [1, \frac{2\gamma}{\gamma+1}]$:
			\[
			\intr |m^\e|^p\,dx \le \lt(\intr \rho^\e|u^\e|^2\,dx\rt)^{\frac p2}\lt(\intr |\rho^\e|^{\frac p{2-p}} \,dx\rt)^{\frac{2-p}2}\le\left(\intrr |v|^2f^\e \,dz\right)^{\frac{p}{2}}\|\rho^\e\|_{L^{\frac{p}{2-p}}}^{\frac{p}{2}}.
			\]
			In this range of $p$, we have $\frac p{2-p} \leq \gamma$, and thus we use $\|\rho^\e\|_{L^1}=1$ and the $L^p$-interpolation inequality to obtain
			\begin{align*}
				\intr |m^\e|^p\,dx & \lesssim \left(\intrr |v|^2f^\e dz\right)^{\frac{p}{2}}\|\rho^\e\|^{\frac{\gamma(p-1)}{\gamma-1}}_{L^\gamma}\\
				&\lesssim \left(\intrr |v|^2f^\e dz\right)^{\frac{p}{2}}\left(\intrr H[f^\e]\,dz\right)^{\frac{p-1}{\gamma-1}}
				\lesssim \left(\intrr H[f^\e]\,dz\right)^{\frac{p}{2}+\frac{p-1}{\gamma-1}},
			\end{align*}
			where we again used \eqref{pres} in the second inequality, thereby obtaining the uniform-in-$\e$ estimate for the momentum $\|m^\e\|_{L^p}$ for $p\in[1,\frac{2\gamma}{\gamma+1}]$. 
		\end{proof}
		
	%%%%%%%%%%%%%%%%%%%%%%%%%%%%%%%%%%%%%%%%%%%%%%
	%
	%
	%
	%
	%
	%
	%%%%%%%%%%%%%%%%%%%%%%%%%%%%%%%%%%%%%%%%%%%%%%
		\subsection{Second spatial moment bound and tightness}
		
		To establish strong compactness for $\rho^\e$ in the whole space $\bbR^d$, we will require uniform bounds on the second spatial moment of $f^\e$.
		
		\begin{lemma}\label{lem_2m}
			Let $f^\e$ be a solution to \eqref{main_kin} with sufficient regularity. Suppose that the potentials $V$ and $K$ satisfy the assumptions  ${\bf (HV)}$ and ${\bf (HK)}$, respectively. Moreover, we assume that the family of initial data satisfy
			\[
				\sup_{\e >0} \left\{\calE[f_0^\e] + \intrr |x|^2 f_0^\e\,dz\right\} <+\infty.
			\] 
			Then, we have
			\[
				\sup_{\e>0}\sup_{t\in[0,T]}\intrr |x|^2 f^\e(t,z)\,dz <+\infty.
			\]
		\end{lemma}
		\begin{remark} If the external potential $V$ is given as quadratic confinement, i.e., $V(x) = |x|^2/2$, then it follows directly from Lemma \ref{lem_energy} that the uniform bound on the second spatial moment holds.
		\end{remark}
		\begin{proof}[Proof of Lemma \ref{lem_2m}]
			%For notational simplicity, we drop the domain of integration. 
            We begin by noting that
			\[\frac{1}{2}\frac{d}{dt}\intrr |x|^2 f^\e \,dz=\frac{1}{\e}\intrr x\cdot v\, f^\e\, dz,\]
			which gives
			\[
				\frac{1}{2}\intrr |x|^2f^\e(t,z)\, dz\le \frac{1}{2}\intrr |x|^2f^\e_0(z)\,dz+\frac{1}{\e}\int_0^t\!\!\intrr x\cdot v\,f^\e(s,z)\,dz\,ds.
			\]
			Thus, it suffices to show that
			\[
				\sup_{\e> 0}\sup_{t\in[0,T]}\left|\frac{1}{\e}\int_0^t\!\!\intrr x\cdot v\, f^\e(s,z)\,dz\,ds\right| <+\infty.
			\]
			On the other hand, we have
			\begin{align*}%\label{xv-est}
				\frac{d}{dt}\intrr x\cdot v\,f^\e dz 
				&=\intrr (x\cdot v)\left(-\frac{1}{\e}v\cdot\nabla_x f^\e+\frac{1}{\e}\sfF[\rho^\e]\cdot\nabla_v f^\e +\frac{1}{\e^2}(M_{\gamma}[\rho^\e]-f^\e)\right) dz\\
				&=\frac{1}{\e}\intrr |v|^2f^\e dz - \frac{1}{\e}\intrr x\cdot \sfF[\rho^\e]\,f^\e dz -\frac{1}{\e^2}\intrr x\cdot v\,f^\e dz\cr
                & \eqqcolon I_1 + I_2 + I_3.
			\end{align*}
Due to assumptions ${\bf (HV)}$ and ${\bf (HK)}$, we easily deduce
\[
	I_1 + I_2 \leq \frac{C}{\e}\lt(2 + \intrr |v|^2f^\e dz + \intr V \rho^\e dx + \frac{1}{2}\intr \rho^\e K*\rho^\e dx \rt),
\]
with constant $C=\max\{2,C_V,C_K\}$. With the energy estimate provided by Lemma~\ref{lem_energy}, we obtain
\[
	\frac{d}{dt}\intrr x\cdot v\,f^\e dz \le \frac{C}{\e}\bigl(2+\calE[f_0^\e]\bigr) -\frac{1}{\e^2}\intrr x\cdot v\,f^\e dz.
\]
Then, Gr\"onwall's lemma yields
		\[
			\intrr x\cdot v\,f^\e(t,z)\,dz\le e^{-t/\e^2}\intrr x\cdot v\,f^\e_0(z)\,dz + C\e\bigl(2+\calE[f_0^\e]\bigr),
		\]
which implies  
			\[
			\frac{1}{\e}\int_0^t\!\!\intrr x\cdot v\,f^\e(s,z)\,dz\,ds\le \e\bigl(1 - e^{-t/\e^2}\bigr)\intrr x\cdot v\,f^\e_0(z)\,dz + C\bigl(2+\calE[f_0^\e]\bigr)T,
			\]
			thus concluding the proof.
		\end{proof}
 	
	%%%%%%%%%%%%%%%%%%%%%%%%%%%%%%%%%%%%%%%%%%%%%%
	%
	%
	%
	%
	%
	%
	%%%%%%%%%%%%%%%%%%%%%%%%%%%%%%%%%%%%%%%%%%%%%%	 	
		
	\subsection{Implication of control on dissipation}
This subsection provides an essential estimate implied by the control on the dissipation term as expressed in Lemma~\ref{lem_energy}, which reads
	\begin{equation}\label{mini}
		0\le \frac{1}{\e^2}\int_0^t\!\!\intrr \bigl(H[f]-H[M_\gamma[\rho_f]]\bigr)\,dz\,ds \le \calE[f_0^\e].
	\end{equation}
	To derive a meaningful estimate from \eqref{mini}, we introduce the following relative entropy functional (or Bregman divergence) for $n>0$ defined by
	\begin{align*}
		\mathfrak{b}_{n}(f|g)&\coloneqq \Psi_n(f)-\Psi_n(g)-\Psi_n'(g)(f-g) =\frac{1}{2c^{2/n}_{\gamma,d}}\left(\frac{f^{1+2/n}}{1+2/n}-\frac{g^{1+2/n}}{1+2/n}-g^{2/n}(f-g)\right).
	\end{align*}
	Note that since $\Psi_n$ is a convex function, the relative entropy $\mathfrak{b}_{n}$ is non-negative for $f,g\ge 0$.
	
	\medskip
	The following result provides a relationship between $\mathfrak{b}_{n}$ and $H$.
	
	\begin{lemma}\label{lem:diss_bregman}
	For every $\rho$ and $f$ with $\intr f\,dv = \rho$, 
		\[
			\intr \mathfrak{b}_n(f|M_\gamma[\rho])\,dv \le \intr \bigl(H[f]-H[M_{\gamma}[\rho]]\bigr)\,dv.
		\]
	\end{lemma}
	 \begin{proof}
 				The estimate follows from the following computation for almost every $x\in\bbR^d$:
		\begin{align*}
		\begin{aligned}%\label{cont_Q}
			\intr \mathfrak{b}_n(f|M_{\gamma}[\rho])\,dv&=\intr \Psi_n(f)-\Psi_n(M_{\gamma}[\rho])-\frac{1}{2c^{2/n}_{\gamma,d}}\bigl(M_{\gamma}[\rho]\bigr)^{2/n}(f-M_{\gamma}[\rho])\,dv\\
			&\hspace{-4em}=\intr\Psi_n(f)-\Psi_n(M_{\gamma}[\rho])-\frac{1}{2}\left(b_0\rho^{\gamma-1}-|v|^2\right)_+(f-M_{\gamma}[\rho])\,dv\\
			&\hspace{-4em}=\intr \bigl(H[f]-H[M_{\gamma}[\rho]]\bigr)\,dv 
			-\frac12 \intr\left(b_0\rho^{\gamma-1}-|v|^2\right)_-(f-M_{\gamma}[\rho])\,dv\\
			&\hspace{-4em}\le \intr \bigl(H[f]-H[M_{\gamma}[\rho]]\bigr) dv,
			\end{aligned}
		\end{align*}
		where we used the fact that the support of $M_\gamma[\rho]$ is given by
		\[
		\text{supp}\,M_\gamma[\rho] =  \lt\{v\in\bbR^d : |v|^2\le b_0\rho^{\gamma-1}\rt\}.	\qedhere
		\]
 	\end{proof}

	The following result shows that the relative entropy provides control over the $L^{1+2/n}$-norm.

	\begin{lemma}\label{lem:diss}Let $n\ge 2$ and $f, g \in L^{1+2/n}(\bbR^d \times \bbR^d)$ with $f,g\ge 0$. Then for any measurable set $B\subset \bbR^d$, 
		\[
			\|f-g\|_{L^{1+2/n}(B\times\bbR^d)}^2 \leq \mathsf{c}_n \max\Bigl\{\|f\|_{L^{1+2/n}}^{1-2/n},\|g\|_{L^{1+2/n}}^{1-2/n}\Bigr\}\iint_{B\times\bbR^d} \mathfrak{b}_n(f|g)\,dz,
		\]
		where $\sfc_n > 0$ depends only on $n$ and $d$.
		\end{lemma}
	\begin{proof} Our proof is motivated by \cite[Proposition 3.1]{CCD02}. We first use Taylor's theorem to the function $s\mapsto s^{p}/p$ for $1< p\le 2$, obtaining
		\begin{equation}\label{taylor}
			\frac{1}{p}\left(f^{p}-g^{p}-pg^{p-1}(f-g)\right)=\frac{1}{2(p-1)}\xi^{p-2}|f-g|^2\qquad\text{for $f>0$,\; $g\ge0$},
		\end{equation}
		where $\xi>0$ lies between $f$ and $g$. The case $f\ge0$ can be handled by taking $f_\delta(z)=f(z)+\delta e^{-|z|^2}$ and letting $\delta\to0$ using dominated convergence theorem as in \cite{CCD02}.
		
		Next, for any measurable set $A\subset \R^{d}\times\R^d$, the H\"older inequality gives
		\[
			\iint_A |f-g|^p\,dz =\iint_A \xi^{\frac{p(p-2)}{2}}\xi^{\frac{p(2-p)}{2}}|f-g|^p\, dz\le\left(\iint_A \xi^{p-2}|f-g|^2\,dz\right)^{p/2}\|\xi\|_{L^p(A)}^{(2-p)p/2}.
		\]
		We now use the above inequality for the sets $A_1=\{f>g\}$ and $A_2=\{f\le g\}$ respectively.
		
		On $A_1$, we have $\xi\le f$ and therefore, $\xi^p\ge f^p$. Hence,
		\begin{align*}
			\|f-g\|_{L^p(A_1)}^2 \le\left(\iint_{A_1} \xi^{p-2}|f-g|^2\,dz\right)\|f\|_{L^p}^{2-p}.
		\end{align*}
		On the set $A_2$, we use $\xi\le g$, i.e., $\xi^p\le g^p$, to obtain
		\[
			\|f-g\|_{L^p(A_2)}^2 \le \left(\iint_{A_2} \xi^{p-2}|f-g|^2\,dz\right) \|g\|_{L^p}^{2-p}.
		\]
		Combining the results of both cases and the expression \eqref{taylor}, we obtain for $p=1+2/n$,
		\begin{align*}
			\|f-g\|_{L^p}^2 &\le 2\max\Bigl\{\|f\|_{L^p}^{2-p},\|g\|_{L^p}^{2-p}\Bigr\}\intrr |f-g|^2\xi^{p-2}\,dz \\
			&\le C_n\max\Bigl\{\|f\|_{L^p}^{2-p},\|g\|_{L^p}^{2-p}\Bigr\}\intrr \mathfrak{b}_n(f|g)\,dz,
		\end{align*}
		with constant $C_n\coloneqq 16\, c_{\gamma,d}^{2/n}/n$.
	\end{proof}
	
	We invoke the following result from \cite[Proposition 4.6]{BV05} as a final ingredient before summarizing the uniform-in-$\e$ estimate based on the control of the dissipation term.

		\begin{proposition}\label{prop51}
			Let $\widehat f\colon \bbV\to [0,b_2]$ be the extension of $f\colon \bbR^d\to [0,+\infty)$ (cf.\ \eqref{eq:extension}). Define
			\[
			\widehat F \coloneqq  \iint_{\bbV} \bigl(|v|^2 + I^2\bigr) \bigl|\widehat f-\widehat M_\gamma[\rho_f]\bigr|(v,I) \,\sigma(dI)\,dv
			\]
			and
			\[
			\widehat D \coloneqq \iint_{\bbV} \bigl(|v|^2 + I^2\bigr) \bigl(\widehat f-\widehat M_\gamma[\rho_f]\bigr)(v,I) \,\sigma(dI)\,dv \ge 0.
			\]
			Then, there exists a constant $C_d > 0$ such that
			\[
			\widehat F \le C_d\Bigl(\widehat D + \sqrt{\rho_f^\gamma\widehat D}\Bigr).
			\]
		\end{proposition}
	
	As a consequence of Lemma~\ref{lem_energy}, \ref{lem:diss_bregman}, \ref{lem:diss} and Proposition~\ref{prop51}, we obtain a control of $\e^{-1}Q_\gamma[f^\e]$, which is fundamental in applying the velocity averaging lemma in the next section.
	
	\begin{corollary}\label{cor:diss}
		Let $f^\e$ be a solution to \eqref{main_kin} with sufficient regularity. Then,
		\[
			\sup_{\e>0}\left\{\left\|\frac{1}{\e}Q_\gamma[f^\e]\right\|_{L^2((0,T);L^{1+2/n}(\bbR^d\times\bbR^d))} + \frac{1}{\e}\int_0^T\!\!\!\intrr |v|^2\,|Q_\gamma[f^\e]|(t,z)\,dz dt\right\} <+\infty.
		\]
        In particular, we have for any compact set $K\subset\bbR^d$,
        \[
            \sup_{\e>0} \left\{ \frac{1}{\e}\int_0^T\!\!\!\iint_{K\times\bbR^d} (1+|v|)|Q_\gamma[f^\e](t,z)|\,dzdt \right\} <+\infty.
        \]
	\end{corollary}
\begin{proof}
	For the first estimate, we apply Lemmas \ref{lem_energy}, \ref{lem:diss_bregman} and \ref{lem:diss} to deduce
	\[
		\left\|\frac{1}{\e}Q_\gamma[f^\e]\right\|_{L^2(L^{1+2/n})}^2 \le \frac{C_n}{\e^2} \int_0^T\!\!\!\intrr \bigl(H[f^\e]-H[M_{\gamma}[\rho_{f^\e}]]\bigr)dz \le C_n\calE[f_0^\e].
	\]
	As for the second estimate, we observe that for almost every $(t,x)\in[0,T]\times\bbR^d$,
\begin{align*}
	F^\e(t,x) &\coloneqq \intr|v|^2|f^\e-M_\gamma[\rho^\e]|(t,x,v)\,dv \\
&\le \iint_{\bbV} \bigl(|v|^2+I^2\bigr)\bigl|\widehat f^\e - \widehat M_\gamma[\rho^\e]\bigr|(t,x,v,I) \,\sigma(dI)\,dv.
\end{align*}
Moreover, due to \eqref{diss_eq}, we have that
\begin{align*}
	D^\e(t,x) &\coloneqq \intr \lt(H_\gamma[f^\e] - H_\gamma[M_\gamma[\rho^\e]]\rt)(t,x,v)\, dv \\
	&= \frac12\iint_{\bbV} \bigl(|v|^2 + I^2\bigr) \bigl(\widehat f^\e -\widehat M_\gamma[\rho^\e]\bigr)(t,x,v,I) \,\sigma(dI)\,dv.
\end{align*}
By Proposition~\ref{prop51}, it then follows that
\[
	F^\e \le C_d\Bigl(D^\e + \sqrt{(\rho^\e)^\gamma D^\e}\Bigr).
\]
Consequently, we obtain
\begin{align*}
	\|F^\e\|_{L^1((0,T)\times\bbR^d)}
	\le C_d\Bigl(\|D^\e\|_{L^1((0,T)\times\bbR^d)}+ \sqrt{\|\rho^\e\|_{L^\gamma((0,T)\times\bbR^d)}^\gamma\|D^\e\|_{L^1((0,T)\times\bbR^d)}}\Bigl).
\end{align*}
Using Lemma \ref{lem_energy}, we deduce
\[
	\frac{1}{\e}\int_0^T\!\!\!\intrr |v|^2\,|Q_\gamma[f^\e]|(t,z)\,dz dt \le C_d\bigl(\e + \|\rho^\e\|_{L^\gamma((0,T)\times\bbR^d)}^{\gamma/2}\bigr).
\]
The minimization principle (cf.\ Lemma~\ref{lem:minimization}) and energy estimate (cf.\ Lemma~\ref{lem_energy}) then gives the required uniform esimate for $\|\rho^\e\|_{L^\gamma((0,T)\times\bbR^d)}$.

As for the final estimate, we simply compute to obtain
\begin{align*}
    \frac{1}{\e}\int_0^T\!\!\!\iint_{K\times\bbR^d} (1+|v|)|Q_\gamma[f^\e](t,z)|\,dzdt& \le 2\sqrt{T}|K\!\times\! B_1|^{\frac{1}{1+n/2}}\left\|\frac{1}{\e}Q_\gamma[f^\e]\right\|_{L^2(L^{1+2/n})} \\
    &\qquad+ \frac{1}{\e}\int_0^T\!\!\!\intrr |v|^2\,|Q_\gamma[f^\e]|(t,z)\,dz dt,
\end{align*}
therewith concluding the proof.
\end{proof}

%%%%%%%%%%%%%%%%%%%%%%%%%%%%%%%%%%%%%%%%%%%%%%
%
%
%
%
%
%
%%%%%%%%%%%%%%%%%%%%%%%%%%%%%%%%%%%%%%%%%%%%%%
\section{Strong compactness}\label{sec_str}
 		
In this section, we deduce the strong compactness of $\rho^\e$ by employing a velocity averaging lemma. In our context, this lemma takes the following form as an adaptation of the velocity averaging lemma found in \cite[Lemma 4.2]{MT07} (see also \cite[Appendix B]{MT07}), and its proof is postponed to Appendix~\ref{sec:velocity}.

	\begin{lemma}\label{lem:averaging}
		Let $p\in(1,+\infty)$ and $f,g_0,g_1\in L^p((0,T);L^p_{\textup{loc}}(\bbR^d\times\bbR^d))$ satisfy the kinetic equation
		\begin{equation}\label{eq:vel-avg}
			\e\pa_t f + v\cdot\nabla_x f = g_0+\nabla_v\cdot g_1.
		\end{equation}
		Then for any $\phi,\psi\in \calC^\infty_c(\R^d)$,
		\[
			(\rho_\psi\phi)(t,x) \coloneqq \intr f(t,x,v)\phi(x)\psi(v)\,dv\quad\in L^p((0,T); B^{(p-1)/2p}_{p,\infty}(\bbR^d)),
		\]
		satisfying the estimate
		\[
            \|\rho_\psi\phi\|_{L^p((0,T);B^{(p-1)/2p}_{p,\infty}(\bbR^d))} \le C\|\phi\|_{W^{1,\infty}}\|\psi\|_{L^1\cap W^{1,\infty}},
        \]
        where the constant $C>0$, depends on the $L^p$-norms $f$, $g_0$ and $g_0$ but is independent of $\phi$ and $\psi$.
	\end{lemma}
		
	Notice that $f^\e$ satisfies the equation \eqref{eq:vel-avg} with
		\begin{equation*}
			g_0^\e = \frac{1}{\e}Q_\gamma[f^\e],\qquad g_1^\e  = \sfF[\rho^\e]f^\e.
		\end{equation*}
		Therefore, to apply Lemma~\ref{lem:averaging}, we will need to provide uniform bounds for $g_0^\e$ and $g_1^\e$.
		
		\begin{lemma}\label{lem:uniform-g}
			Let $f^\e$ be a weak entropy solution to \eqref{main_kin} and $n \geq 2$. Suppose that $V$ and $K$ satisfy the assumptions  ${\bf (HV)}$ and ${\bf (HK)}$, respectively. Then,
			\[
				\sup_{\e>0} \left\|\frac{1}{\e}Q_\gamma[f^\e]\right\|_{L^p((0,T);L^p_{loc}(\bbR^d \times \bbR^d))} <+\infty \qquad \text{for every } p \in \lt[1, 1+\frac 2n\rt]
			\]
			and  
\[
				\sup_{\e>0} \|\sfF[\rho^\e]f^\e\|_{L^\ell([0,T];L^\ell_{loc}(\bbR^d \times \bbR^d))} <+\infty  \quad \mbox{for} \quad \frac1\ell = \frac1{1+\frac2n} + \frac1r \in \lt[\frac1{1+\frac2n}, \ 1 \rt).
\]
		\end{lemma}
		\begin{proof}
			The first estimate can be easily deduced from Corollary~\ref{cor:diss} by using the monotonicity of the $L^p$-norm on bounded domains.
			
			For the potential $V \in W^{1,\infty}_{loc}(\bbR^d)$, we deduce for all $\ell\in(1,1+2/n]$ that
			\[
				\sup_{\e>0}\| \nabla V f^\e\|_{L^\ell([0,T];L^\ell_{loc}(\bbR^d \times \bbR^d))} \leq \|\nabla V\|_{L^\infty_{loc}(\bbR^d)}\sup_{\e>0}\|f^\e\|_{L^\ell([0,T];L^\ell_{loc}(\bbR^d \times \bbR^d))} <+\infty.
			\]
			For the term with the interaction potential $K$, we use \eqref{condi_hk2} to find $\ell > 1$ and $s \in [1,\gamma]$ satisfying
			\[
			\frac1\ell = \frac1{1+\frac2n} + \frac1r \quad \mbox{and} \quad 1 + \frac1r = \frac1q + \frac1s.
			\]
			Then, we apply H\"older's inequality and Young's convolution inequality to deduce
			\begin{align*}
		\|(\nabla K*\rho^\e)f^\e\|_{L^\ell_{loc}} &\leq C\|\nabla K * \rho^\e\|_{L^r}\|f^\e\|_{L^{1+2/n}} \cr
		&\leq C\lt( \|\nabla K\|_{L^q(B_R)}\|\rho^\e\|_{L^s} + \|\nabla K\|_{L^r(\bbR^d \setminus B_R)}\|\rho^\e\|_{L^1} \rt)\|f^\e\|_{L^{1+2/n}} <+ \infty
\end{align*}
uniformly in $\e>0$ due to ${\bf (HK)}$. 	
	\end{proof}
		
In conjunction with an Aubin--Lions-type lemma \cite{Au63,Li69,RoSa03,Si87}, we establish a strong compactness result for the sequence $\{\rho^\e\}_{\e>0}$ in the following lemma.
 
\begin{lemma}
	Let $f^\e$ be a weak entropy solution to \eqref{main_kin}. Suppose $V$ and $K$ satisfy the assumptions {\bf (HV)} and {\bf (HK)}, respectively. Then the family
    \[
        \{\rho^\e\}_{\e>0}\;\;\text{is relatively compact in $L^q((0,T)\times\bbR^d)$,\; $q\in[1,\gamma)$,}
    \]
    i.e., there exists $\rho\in L^q((0,T)\times\bbR^d)$ and a (not relabelled) subsequence of $\{\rho^\e\}_{\e>0}$ such that
	\[
		\rho^\e \to \rho\quad\text{in $L^q((0,T)\times\bbR^d)$}.
	\]
	Moreover, $\rho\in \calC([0,T], W^{-1,1}(\bbR^d))$ and
	\[
		\rho^\e(t,\cdot)\,\calL^d \rightharpoonup \rho(t,\cdot)\,\calL^d\quad\text{narrowly in $\calP(\bbR^d)$}\quad\text{for every $t\in[0,T]$,}
	\]
	where $\calL^d$ is the $d$-dimensional Lebesgue measure on $\bbR^d$.
\end{lemma}
\begin{proof}
    The proof will be performed in three steps. In the first step, we show that the density $\rho^\e$ belongs to appropriate Besov space by using Lemma \ref{lem:averaging} and the second-moment bounds of $f^\e$ in Section \ref{sec_uni}. Then, in the second step, we will show that $\{\rho^\e\}_{\e>0}$ is bounded in $W^{1,1}((0,T);W^{-1,1}(\R^d))$. Finally, we obtain the compactness of $\{\rho^\e\}_{\e>0}$ in $L^q((0,T)\times\R^d)$ for $1\le q<\gamma$ using an Aubin--Lions-type lemma and the interpolation between $L^1$ and $L^\gamma$.
    
    \medskip
    \noindent{\em Step 1}. We deduce from Proposition~\ref{prop_uni} and Lemma~\ref{lem:uniform-g} that the sequences $\{f^\e\}_{\e>0}$, $\{g_0^\e\}_{\e>0}$, $\{g_1^\e\}_{\e>0}$ are bounded in $L_{loc}^p((0,T)\times\bbR^d\times\bbR^d)$ for some $p>1$. Hence, we deduce from Lemma~\ref{lem:averaging} that, for any $\phi,\psi\in \calC_c^\infty(\bbR^d)$, the family
    \[
        \{\phi\rho^\e_\psi\}_{\e>0} \subset L^p((0,T); B^{(p-1)/2p}_{p,\infty}(\bbR^d))\quad\text{is bounded,}
    \]
    with $B^{(p-1)/2p}_{p,\infty}(\bbR^d)\hookrightarrow L^p(\bbR^d)$ being a compact embedding. 
    
    We first focus on removing $\phi$. To this end, we choose a sequence of smooth cutoff radial function $\phi_k$ such that $\phi_k(x)=1$ for $|x|\le k$ and $\phi_k(x)=0$ for $|x|\ge 2k$, satisfying additionally $\|\nabla\phi_k\|_{L^\infty}\le 2k^{-1}$ for all $k\ge 1$. We start by splitting $\Delta_h \rho^\e_\psi$ and estimating
    \[\|\Delta_h\rho^\e_\psi\|_{L_x^p} \leq\|\Delta_h (\phi_k\rho^\e_\psi)\|_{L_x^p}+\|\Delta_h[(1-\phi_k)\rho^\e_\psi]\|_{L_x^p}.\]
    Since the estimate of Lemma \ref{lem:averaging} depends linearly on $\|\phi_k\|_{W^{1,\infty}}$, there exists $C(t)\in L^p((0,T))$, independent of $k\ge 1$, such that
    \[
    \|\Delta_h(\phi_k\rho^\e_\psi)\|_{L_x^p}\le C(t) |h|^{\frac{p-1}{2p}}\|\psi\|_{L^1\cap W^{1,\infty}}.
    \]
    Furthermore, since
    \[\Delta_h(fg)=(\Delta_hf)g(\cdot+h)-f\Delta_hg,\]
    and $\|\Delta_h\phi_k\|_{L^\infty}\le \|\nabla\phi_k\|_{L^\infty}|h|\le 2k^{-1}|h|$, we control the second term as
    \begin{align*}
        \|\Delta_h[(1-\phi_k)\rho^\e_\psi]\|_{L^p}&=\|\Delta_h\phi_k \rho^\e_\psi(\cdot+h)\|_{L_x^p}+\|(1-\phi_k)\Delta_h \rho^\e_\psi\|_{L_x^p}\\
        &\le 2|h|k^{-1}\|\rho^\e_\psi\|_{L_x^p}+\|(1-\phi_k)\Delta_h\rho^\e_\psi\|_{L_x^p}.
    \end{align*}
    Finally, we estimate the second term on the right-hand side. The estimate of this term is essentially the same as controlling the term
    \[
    	\int_{|x|\ge k}(\rho^\e)^p\,dx.
    \]
    Since, for some $\alpha\in(0,2)$, Young's inequality implies
    \[
    	\int_{|x|\ge k}(\rho^\e)^p\,dx\le \frac{1}{k^{\alpha}}\int_{\R^d}|x|^{\alpha}(\rho^\e)^p\,dx\le\frac{1}{k^{\alpha}}\int_{\R^d}\left(|x|^2\rho^\e + (\rho^\e)^{(p-\frac{\alpha}{2})\frac{2}{2-\alpha}}\right)\,dx,
    \]
    we may choose $p<\gamma$ and $\alpha$ sufficiently small such that
    \[
    	0<\alpha\le \frac{2(\gamma-p)}{\gamma-1}\quad\Longrightarrow\quad 1\le \left(p-\frac{\alpha}{2}\right)\frac{2}{2-\alpha}\le \gamma.
    \]
    This, together with the second moment bound of $f^\e$ in $x$-variable, implies
    \[\|(1-\phi_k)\Delta_h\rho^\e_\psi\|_{L_x^p}\le Ck^{-\alpha/p}.\]
    Gathering all the estimates and letting $k\to\infty$, we arrive at the estimate
    \[
        \|\Delta_h\rho^\e_\psi\|_{L_x^p} \le C(t)|h|^{\frac{p-1}{2p}}\|\psi\|_{L^1\cap W^{1,\infty}}\qquad\text{for any $h>0$.}
    \]
    
    Now, for each fixed $h>0$, we again consider a smooth radially symmetric cutoff function $\psi_k$ such that $|\psi_k(v)|=1$ for $|v|\le k$ and $|\psi_k(v)|=0$ for $|v|\ge k+1$. Then, thanks to the geometric structure of $\psi_k$, we easily deduce that for sufficiently large $k\gg 1$,
    \[\|\psi_k\|_{L^1 \cap W^{1,\infty}}\le Ck^d,\]
    where $C>0$ is independent of $k$.
    Now, for some $s>0$, 
    \begin{align*}
    |h|^{-s}\|\Delta_h \rho^\e\|_{L^p} &\le |h|^{-s}\|\Delta_h\rho^\e_{\psi_k}\|_{L^p}+|h|^{-s}\left\|\int_{\R^d}(f^\e(x+h,v)-f^\e(x,v))(1-\psi_k(v))\,dv\right\|_{L^p}\\
    &\le C(t)|h|^{-s}|h|^{\frac{p-1}{2p}}\|\psi_k\|_{L^1 \cap W^{1,\infty}} + |h|^{-s} k^{-1-\alpha}\\
    &\le C(t)|h|^{-s+\frac{p-1}{2p}}k^d+|h|^{-s} k^{-1-\alpha}
    \end{align*}
    for some $\alpha\in(0,1)$, where we use the boundedness of the second velocity moment of $f^\e$. We now choose $\beta>0$ and $s>0$ according to
    \[
       \beta = \frac{p-1}{2p}\frac{1}{(1+\alpha+d)},\qquad s = \beta(1+\alpha),
    \]
    and take $k=h^{-\beta}$ to obtain
    \[
        |h|^{-s}\|\Delta_h \rho^\e\|_{L^p}\le C(t)|h|^{-s+\frac{p-1}{2p}-d\beta}+|h|^{-s+\beta(1+\alpha)}= C(t)+1.\]
    Therefore, we conclude that
    \[
        \sup_{h\ne 0}|h|^{-s}\|\Delta_h\rho^\e\|_{L^p}\in L^p((0,T)),
    \]
    that is, $\rho^\e$ is bounded in $L^p(0,T; B^s_{p,\infty}(\R^d))$ for the above choice of $s$. 
    
    Finally, using the second moment bound, we obtain
    \begin{align*}
		\|\Delta_h\rho^\e\|_{L^1} &\le |B_R|^{1-1/p}\|\Delta_h\rho^\e\|_{L^p} + \int_{B_R^c} |\Delta_h\rho^\e(x)|\,dx \le C(t)R^{d(1-1/p)}|h|^{s} + CR^{-2}\\
        &\le C(t)|h|^{2s/(d(1-1/p)+2)},
    \end{align*}
    where we use optimal choice $R=|h|^{-s/(d(1-1/p)+2)}$ in the last inequality. Therefore, we conclude that $\{\rho^\e\}_{\e>0}$ is bounded in $L^p((0,T);B^{s'}_{1,\infty}(\R^d))$ for $s'=2s/(d(1-1/p)+2)$.
    
    \medskip
    \noindent{\em Step 2}. We now show that $\{\rho^\e\}_{\e>0}\subset W^{1,1}((0,T);W^{-1,1}(\bbR^d))$ is bounded. For any $\varphi\in \calC_c^\infty((0,T)\times\bbR^d)$, we have that $\rho^\e$ satisfies
	\[
		\int_0^T\!\!\!\intr \e\partial_t\varphi\,\rho^\e + \nabla\varphi\cdot m^\e\,dx dt = 0,\qquad m^\e(t,x)\coloneqq \intr v\,f^\e(t,x,v)\,dv. 
	\]
	Taking $\varphi(t,x) = \chi(t) \phi(x)$ with $\chi\in \calC_c^\infty((0,T))$ and $\phi\in \calC_c^\infty(\bbR^d)$, we find
	\[
		\int_0^T \!\! \dot\chi(t)\intr\rho^\e(t,x)\,\phi(x)\,dx dt = -\frac{1}{\e}\int_0^T\!\!\chi(t)\!\int_K \nabla\phi(x)\cdot m^\e(t,x)\,dx dt.
	\]
	From the zero $v$-mean property of $M_\gamma[\rho^\e]$, we deduce the estimate
	\begin{align*}
		\left\|\frac{1}{\e}\nabla\phi\cdot m^\e\right\|_{L^1} &= \int_0^T\!\!\!\intr \frac{1}{\e}\left|\nabla\phi(x)\cdot\intr v\,f^\e(t,x,v)\,dv\right|dxdt \\
		&= \int_0^T\!\!\!\intr \frac{1}{\e}\left|\nabla\phi(x)\cdot\intr v\,Q_\gamma[f^\e](t,x,v)\,dv\right|dxdt \\
		&\le \|\nabla\phi\|_{L^\infty}\frac{1}{\e}\int_0^T\!\!\!\iint_{(\text{supp}\, \phi)\times\bbR^d} |v||Q_\gamma[f^\e](t,z)| dzdt.
	\end{align*}
    Corollary~\ref{cor:diss}, then implies that 
	\[
		\left\{\frac{1}{\e}\nabla\phi\cdot m^\e\right\}_{\e>0}\subset L^1((0,T)\times \bbR^d).
	\]
	Consequently, we have that the sequence $
		\{t\mapsto \langle\phi,\rho^\e(t,\cdot)\rangle\}_{\e>0}$ is bounded in $W^{1,1}((0,T))$ for every $\phi\in W^{1,\infty}_0(\bbR^d)$, implying that $\{\rho^\e\}_{\e>0}$ is bounded in $W^{1,1}((0,T);W^{-1,1}(\bbR^d))$. Together with the uniform-bound $\|\rho^\e(t,\cdot)\|_{L^1}\le 1$, we obtain the pointwise-in-time narrow convergence
	\[
		\rho^\e(t,\cdot)\,\calL^d \rightharpoonup \rho(t,\cdot)\,\calL^d\quad\text{narrowly in $\calP(\bbR^d)$}\quad\text{for every $t\in[0,T]$,}
	\]
	hold by applying a generalized Arzel\'{a}--Ascoli theorem \cite{AGS08}.
		
	\medskip
	\noindent{\em Step 3}. Using the fact that $B^{s'}_{1,\infty}(\R^d)$ embeds compactly into $L^1(\bbR^d)$, we combine the results of Steps 1 \& 2 to use an Aubin-Lions-type lemma \cite{Si87} to conclude that $\{\rho^\e\}_{\e>0}$ is relatively compact in $L^1((0,T)\times\R^d)$. Finally, the uniform boundedness of $\rho^\e$ in $L^\infty((0,T);L^\gamma(\R^d))$ and the interpolation inequality yield the desired compactness of $\rho^\e$ in $L^q((0,T)\times\R^d)$ for all $1\le q<\gamma$.
\end{proof}
	 		
		%%%%%%%%%%%%%%%%%%%%%%%%%%%%%%%%%%%%%%%%%%%%%%
		%
		%
		%
		%
		%
		%
		%%%%%%%%%%%%%%%%%%%%%%%%%%%%%%%%%%%%%%%%%%%%%%
\section{Passing to the limit $\e\to 0$: Proof of Theorem~\ref{thm_main}}\label{sec_pf}
		
This section provides a detailed proof of Theorem \ref{thm_main}, outlining the weak formulations of the equations and the convergence analysis as $\e \to 0$.
		%%%%%%%%%%%%%%%%%%%%%%%%%%%%%%%%%%%%%%%%%%%%%%
	%
	%
	%
	%
	%
	%
	%%%%%%%%%%%%%%%%%%%%%%%%%%%%%%%%%%%%%%%%%%%%%%
\subsection{Weak formulations}
We begin by recalling the moment equations for $f^\e$ given in \eqref{eq:moments}, which are written as
		$$\begin{aligned}
			\pa_t \rho^\e + \frac{1}{\e}\nabla \cdot  m^\e &= 0,\cr
			\e\pa_t m^\e + \nabla \cdot \intr v \otimes v (f^\e - M_\gamma[\rho^\e])\,dv  &= -\frac{1}{\e} m^\e - \nabla (\rho^\e)^\gamma - \rho^\e \sfF[\rho^\e],
		\end{aligned}$$
		where we used the identity
		\[
		\nabla \cdot \intr v \otimes v M_\gamma[\rho^\e]\,dv = \nabla (\rho^\e)^\gamma.
		\]
		
For any test functions $\varphi\in \calC^\infty_c([0,T)\times\bbR^d)$, $\xi\in \calC^\infty_c([0,T)\times\bbR^d;\bbR^d)$, their weak formulations are
		\[
			\intr \varphi(0,x)\rho^\e_0(x)\,dx + \int_0^T \!\!\!\intr \left(\partial_t\varphi\,\rho^\e + \frac{1}{\e}\nabla\varphi\cdot m^\e\right) dxdt = 0
			\]
			and
					\begin{align*}
			&\intr \e\,\xi(0,x) m^\e_0(x)\,dx + \int_0^T \!\!\!\intr \e\,\partial_t \xi\, m^\e\,dxdt + \int_0^T \!\!\!\intrr (\nabla \xi): (v\otimes v) \bigl( f^\e - M_\gamma[\rho^\e]\bigr)\,dzdt \\
			&\hspace{8em} = \int_0^T\!\!\!\intr -(\rho^\e)^\gamma\,\text{Tr}(\nabla\xi) + \xi\cdot \rho^\e \sfF[\rho^\e] + \frac{1}{\e}\,\xi\cdot m^\e\,dxdt.
		\end{align*}
Choosing $\xi=\nabla\varphi$, the two equations can be combined to obtain
		\begin{align}\label{lim_weakf}
			\begin{aligned}
				&\intr \varphi(0,x)\rho^\e_0(x)\,dx + \int_0^T\!\!\!\intr \partial_t\varphi\,\rho^\e + \Delta\varphi\,(\rho^\e)^\gamma - \nabla\varphi \cdot \rho^\e\, \sfF[\rho^\e]\,dxdt \cr
				&\hspace*{8em} = - \intr \e\nabla\varphi(0,x) m^\e_0(x)\,dx - \int_0^T \!\!\!\intr \e\,\partial_t \nabla\varphi \cdot  m^\e\,dx dt \\
				 &\hspace*{12em}  + \int_0^T \!\!\!\intrr (\nabla^2 \varphi) \!:\!  (v\otimes v)\,\bigl( f^\e - M_\gamma[\rho^\e]\bigr)\,dz\,dt. 
			\end{aligned}
		\end{align}

		%%%%%%%%%%%%%%%%%%%%%%%%%%%%%%%%%%%%%%%%%%%%%%
	%
	%
	%
	%
	%
	%
	%%%%%%%%%%%%%%%%%%%%%%%%%%%%%%%%%%%%%%%%%%%%%%
\subsection{Passing to the limit $\e \to 0$}
				
To analyze the convergence, we use the results from Section \ref{sec_str}, which state: For every $q \in [1,\gamma)$,
\begin{gather}\label{con_rho_f}
	\begin{gathered}
	\rho^\e \rightharpoonup \rho \quad \mbox{in } L^\gamma((0,T) \times \bbR^d) \quad \mbox{and} \quad \rho^\e \to \rho \quad \mbox{in } L^q((0,T) \times \bbR^d) \mbox{ and a.e.},\\
			\rho^\e(t,\cdot)\,\calL^d \rightharpoonup \rho(t,\cdot)\,\calL^d\quad\text{narrowly in $\calP(\bbR^d)$}\quad\text{for every $t\in[0,T]$.}
	\end{gathered}
\end{gather}
For the initial data, the assumed convergence gives
\[
\intr \varphi(0,x)\rho^\e_0(x)\,dx \to \intr \varphi(0,x)\rho_0(x)\,dx \quad \mbox{as} \quad \e \to 0.
\]
Using \eqref{con_rho_f}, we easily  get
	\[
	\int_0^T\!\!\!\intr \pa_t \varphi \rho^\e \,dxdt\to\int_0^T\!\!\!\intr \pa_t \varphi \rho\,dxdt \quad \mbox{as} \quad \e \to 0.
	\]
The terms involving $m^\e$ vanish in the limit due to the uniform-in-$\e$ bound estimate on $m^\e \in L^\infty([0,T];L^1(\bbR^d))$ (cf.\ Proposition \ref{prop_uni}), and the dissipative estimates in Corollary \ref{cor:diss} imply
 \begin{align*}
	 &\lt|\int_0^T\!\!\! \intrr (\nabla^2 \varphi): (v\otimes v) \bigl( f^\e - M_\gamma[\rho^\e]\bigr)\,dzdt\rt|   \cr
	 &\hspace*{8em} \leq \|\nabla^2 \varphi\|_{L^\infty} \int_0^T\!\!\!\intrr |v|^2|M_\gamma[\rho^\e]- f^\e| \,dzdt \le C_d\e.
 \end{align*}

For the nonlinear diffusion term, we find $(\rho^\e)^\gamma \to \rho^\gamma$ a.e. and $(\rho^\e)^\gamma \to \rho^\gamma$ in $\calD'((0,T) \times \Omega)$ for any compact set $\Omega \subset \bbR^d$ since $\rho^\e \to \rho$ a.e. Consequently,
	\[
	\int_0^T\!\!\!\intr \Delta \varphi\,(\rho^\e)^\gamma\,dxdt\to\int_0^T\!\!\!\intr \Delta \varphi\,\rho^\gamma\,dxdt \quad \mbox{as $\e \to 0$,} 
	\]
since $\varphi$ is compactly supported.

The term involving $V$ is linear and can be easily handled. To analyze the convergence of the nonlocal interaction term, we use the symmetry of $K$ to obtain
	\begin{align*}
		&\intr   \nabla\varphi \cdot (\rho^\e \nabla K * \rho^\e - \rho \nabla K * \rho)\,dx\cr
		&\hspace*{4em} = \intrr \nabla \varphi(x) \cdot (\rho^\e - \rho)(x) \nabla K (x-y) \rho^\e(y)\,dxdy \\
		&\hspace*{4em}\qquad+ \intrr \nabla \varphi(x) \cdot \rho(x) \nabla K(x-y) (\rho^\e - \rho)(y)\,dxdy\cr
		&\hspace*{4em} = \intrr (\rho^\e - \rho)(x) \nabla K(x-y) \cdot (\nabla \varphi(x) \rho^\e(y) - \nabla \varphi(y) \rho(y))\,dxdy\cr
                 &\hspace*{4em} = \frac12\intrr (\rho^\e - \rho)(x) \nabla K(x-y) \cdot (\nabla \varphi(x) - \nabla \varphi(y))(\rho^\e + \rho)(y)\,dxdy\cr
                 &\hspace*{4em} = \frac{1}{2}\bigl(I^\e + II^\e\bigr),
	\end{align*}
where 
\[
I^\e := \iint_{|x-y| \leq R} (\rho^\e - \rho)(x) \nabla K(x-y) \cdot (\nabla \varphi(x) - \nabla \varphi(y))(\rho^\e + \rho)(y)\,dxdy,
\]
\[
II^\e := \iint_{|x-y| \geq R} (\rho^\e - \rho)(x) \nabla K(x-y) \cdot (\nabla \varphi(x) - \nabla \varphi(y))(\rho^\e + \rho)(y)\,dxdy.
\]
For $I^\e$, we estimate
\begin{align*}
I^\e &\leq \|\nabla^2 \varphi\|_{L^\infty} \iint_{|x-y| \leq R} |(\rho^\e - \rho)(x)| |(x-y)\cdot\nabla K(x-y)| (\rho^\e + \rho)(y)\,dxdy\cr
&\leq C_K  \|\nabla^2 \varphi\|_{L^\infty} \iint_{|x-y| \leq R} |(\rho^\e - \rho)(x)|(1 +  |K(x-y)|) (\rho^\e + \rho)(y)\,dxdy\cr
&\leq C_K  \|\nabla^2 \varphi\|_{L^\infty} \lt( 2\|\rho^\e - \rho\|_{L^1} +  \iint_{|x-y| \leq R} |(\rho^\e - \rho)(x)| |K(x-y)| (\rho^\e + \rho)(y)\,dxdy\rt)\cr
&=: I_1^\e + I_2^\e.
\end{align*}
It is clear that $I_1^\e \to 0$ as $\e \to 0$. For $I^\e_2$, using the first condition in \eqref{condi_hk2}, i.e., $\frac1p < 2 - \frac2\gamma$, we find indices $\ell \in [1,\gamma)$ and $s \in [1,\gamma]$ such that
\[
\frac1\ell + \frac1s = 2 - \frac1p.
\]
Applying Young's convolution inequality, we then obtain
\[
I_2^\e \leq C_K \|\nabla^2 \varphi\|_{L^\infty} \|K\|_{L^p(B_{2R})}\|\rho^\e - \rho\|_{L^\ell} \|\rho^\e + \rho\|_{L^s} \to 0 \quad \mbox{as} \quad \e \to 0.
\]
A similar argument holds for $II^\e$. Using the condition $0 \leq \frac1r < 2 - \frac2\gamma$, we find $\ell \in [1,\gamma)$ and $s \in [1,\gamma]$ such that
\[
\frac1\ell + \frac1s = 2 - \frac1r.
\]
We then estimate
\[
II^\e \leq C_K \|\nabla \varphi\|_{L^\infty} \|\nabla K\|_{L^p(\R^d \setminus B_R)}\|\rho^\e - \rho\|_{L^\ell} \|\rho^\e + \rho\|_{L^s} \to 0 \quad \mbox{as} \quad \e \to 0.
\]
 
	To conclude, we pass to the limit $\e \to 0$ in \eqref{lim_weakf} to obtain
	\[
\intr \varphi(0,x)\rho_0(x)\,dx + \int_0^T \!\!\!\intr \partial_t\varphi\,\rho +\Delta\varphi\,\rho^\gamma - \nabla\varphi \cdot \rho \sfF[\rho]\,dxdt = 0,
	\]
	which implies the limit $\rho$ is the weak solution to
	\[
		\pa_t \rho = \Delta \rho^\gamma + \nabla\cdot(\rho \sfF[\rho]).
	\]
		%%%%%%%%%%%%%%%%%%%%%%%%%%%%%%%%%%%%%%%%%%%%%%
	%
	%
	%
	%
	%
	%
	%%%%%%%%%%%%%%%%%%%%%%%%%%%%%%%%%%%%%%%%%%%%%%
	
	As for the uniform-in-time bound of the energy $\calF$, we simply use the lower semicontinuity of $\calF$ w.r.t.\ the narrow convergence in $\calP(\bbR^d)$ and the minimization property of the local equilibrium,
	\[
		\left(1+\frac{d}{2}\right)\|\rho^\e(t,\cdot)\|_{L^\gamma}^\gamma  =\intrr H[M_\gamma[\rho^\e(t,\cdot)]]\,dz \le \intrr H[f^\e(t,\cdot)]\,dz,
	\]
	to obtain
	\[
		\calF[\rho(t,\cdot)] \le \liminf_{\e\to 0} \calF[\rho^\e(t,\cdot)] \le \sup_{\e>0} \calE[f_0^\e] <+\infty,
	\]
	for every $t\in[0,T]$. Taking the supremum in $t\in[0,T]$ yields the required energy bound.

	\subsection{Convergence of $f^\e$}	
	In this part, we show the convergence of $f^\e$ to $M_\gamma[\rho]$. We first notice from Corollary \ref{cor:diss} that
\[%	\bq\label{conv_fe}
	\|f^\e - M_\gamma[\rho^\e]\|_{L^2((0,T);L^{1+ 2/n}(\bbR^d \times \bbR^d))} \leq C\e 
\]%	\eq
	for some $C>0$ independent of $\e>0$. In particular, this implies
\[
	\|f^\e - M_\gamma[\rho^\e]\|_{L^1((0,T);L^1_{loc}(\bbR^d \times \bbR^d))} \to 0 \quad\text{as $\e \to 0$.}
\]	
	On the other hand, it follows from Lemma \ref{lem:nemytskii} that	
\[
\|M_\gamma[\rho^\e]-M_\gamma[\rho]\|_{L^1(\bbR^d \times \bbR^d)} \le a_\gamma\|\rho^\e-\rho\|_{L^1(\bbR^d \times \bbR^d)},
\]	
where $a_\gamma > 0$ is independent of $\e > 0$. Combining these two estimates together with \eqref{con_rho_f} yields
\begin{align*}
\|f^\e - M_\gamma[\rho]\|_{L^1(0,T;L^1_{loc}(\bbR^d \times \bbR^d))} 
\;\longrightarrow\; 0  \quad\text{as $\e \to 0$},
\end{align*}
thus completing the proof.
	
	%%%%%%%%%%%%%%%%%%%%%%%%%%%%%%%%%%%%%%%%%%%%%%
	%
	%
	%
	%
	%
	%
	%%%%%%%%%%%%%%%%%%%%%%%%%%%%%%%%%%%%%%%%%%%%%%
	\section{Existence of solutions: Proof of Theorem~\ref{thm_kext}}\label{sec_exi}
	
	This section aims to prove the existence of weak entropy solutions to the BGK-type model
	\begin{equation}\label{eq:extend-exist}
		\e\pa_t f^\e+v\cdot \nabla f^\e -\sfF[\rho^\e]\cdot\nabla_v f^\e = \frac{1}{\e}\bigl(M_\gamma[\rho^\e]-f^\e\bigr),
	\end{equation}
	subject to the initial data $f(0,\cdot) = f_0$ in $\frkX^\gamma$, i.e., we are concerned with Theorem~\ref{thm_kext}. As before, the force is given by
	\[
		\sfF[\rho] = \nabla V + \nabla K\ast\rho,
	\]
	with an external potential $V$ and interaction potential $K$ satisfying the assumptions {\bf(HV)} and {\bf(HK)} respectively. %In this section, we further assume w.l.o.g.\ that $V, K\ge 0$.
	
	The proof of Theorem~\ref{thm_kext} follows by approximation from a sequence of mild solutions to \eqref{eq:extend-exist} for regular forces $\sfF$ (cf.\ Theorem~\ref{thm:mild} below). These mild solutions are obtained via a fixed-point argument by Leray--Schauder (see eg. \cite[Chapter 6.8]{Zeidler}).
	
	\subsection{Preliminary results}
	We begin by discussing the ingredients of the fixed-point operator defined in the following section.
	
	We recall the map $m_\gamma:[0,\infty)\times \bbR^d \to[0,\infty)$ defined for $\gamma \in (1, 1 +\frac{2}{d+2}]$ by
	\[
		m_\gamma(\rho,v) = c_{\gamma,d}\bigl(b_0\rho^{\gamma-1}-|v|^2\bigr)^{n/2}_+,\qquad b_0 = \frac{2\gamma}{\gamma-1},
	\]
	and the Nemytskii operator $M_\gamma[\rho](z) = m_\gamma(\rho(x),v)$, $z=(x,v)$. %We further set
	%\begin{gather*}
	%	\calR^\gamma \coloneqq L^1\cap L^\gamma(\bbR^d),\qquad \calX^\gamma \coloneqq L^1\cap L^{1+2/n}(\bbR^d\times\bbR^d).
	%\end{gather*}

	\subsubsection{Lagrangian flow}
	An important ingredient we will need is the Lagrangian flow $\Phi_t^\rho(s,z)$ defined for $s,t\in[0,T]$ and $z\in\R^{d}\times\bbR^d$, as follows: Let $Z_t \coloneqq (X_t,V_t)$ satisfy
	\begin{align*}
		\frac{d}{dt} X_t = \frac{1}{\e}V_t,\qquad \frac{d}{dt} V_t = - \frac{1}{\e}\sfF[\rho_t](X_t),
	\end{align*}
	with the condition $(X_s,V_s)=(x,v)=z$, $s\in[0,T]$. We then set $\Phi_t^\rho(s,z)\coloneqq Z_t$ for $t\in[0,T]$.
	
	We will restrict ourselves for the moment to the class of regular forces $\sfF$ satisfying
	\begin{align}\label{ass:flow}
		\left\{\quad\begin{gathered}
		\esssup_{t\in[0,T]}\|\sfF[\eta_t]\|_{\calC_b^1} \le c_\sfF\|\eta\|_{L^\infty(L^1)}\qquad \forall\eta\in L^\infty((0,T);L^1(\bbR^d)),\\
		\esssup_{t\in[0,T]}\|\sfF[\rho_t]-\sfF[\eta_t]\|_{\infty} \to 0\qquad\text{as}\qquad \|\rho-\eta\|_{L^\infty((0,T);L^1(\bbR^d))} \to 0. 
		\end{gathered}\right.\tag{A$_\sfF$}
	\end{align}	
	\begin{lemma}%\label{lem:flow}
		Let $\rho\in L^\infty((0,T);L^1(\bbR^d))$ and $\sfF$ satisfy \eqref{ass:flow}. Then a conservative flow exists globally, i.e., $\det D\Phi_t^\rho(s,z)=1$ for all $s,t\in[0,T]$ and $z\in\bbR^d\times\bbR^d$. Moreover,
		\[
		\begin{gathered}
			|\Phi_t^{\rho}(s,z)| \le |z|e^{c_\sfF\|\rho\|_{L^\infty(L^1)}|t-s|/\e},\qquad
%			|\Phi_t^{\rho,1}(s,z)| \le |x| + |v||t-s| + \e^{-1}c_\sfF\|\rho\|_{L^1} |t-s|^2,\\
%			|\Phi_t^{\rho,2}(s,z)| \le |v| + \e^{-1}c_\sfF\|\rho\|_{L^1} |t-s|,\\
			\|D\Phi_t^\rho(s,z)\|_{\rm F} \le e^{(1+c_\sfF\|\rho\|_{L^\infty(L^1)})|t-s|/\e}.
		\end{gathered}
		\]
	\end{lemma}
		
	\begin{remark}
		The first inequality shows, in particular, that if $z\in B(0,R)$,
		\[
		\Phi_t^{\rho}(s,z)\in B(0,R_t^\rho(s))\qquad\text{with}\qquad R_t^\rho(s)\coloneqq Re^{\e^{-1}c_\sfF\|\rho\|_{L^\infty(L^1)}|t-s|}.
		\]
		Therefore, if $h\in \frkX^\gamma$ with $\supp h \subset B(0,R)$, then
		\[
			h(\Phi_s^\rho(t,z)) >0 \quad\Longleftrightarrow\quad \Phi_s^\rho(t,z) \in B(0,R) \quad\Longleftrightarrow\quad z\in \Phi_t^\rho(s,B(0,R)),
		\]
		thus implying that
		\[
			h(\Phi_s^\rho(t,z)) = 0 \qquad\text{for all $z\notin \overline{B(0,R_t^\rho(s))}$}.
		\]
	\end{remark}
	
	\begin{lemma}\label{lem:lagrangian-stability}
		Let $\rho,\eta\in L^\infty((0,T); L^1(\bbR^d))$. Then,
		\begin{align*}
			|\Phi_t^{\rho}(s,z) - \Phi_t^{\eta}(s,z)| 
			%	&\le \frac{1}{\e}e^{2(2+c_K)t/\e}\int_s^t e^{-2(2+c_K)r/\e}\,ds\,\|\nabla K\ast(\rho-\eta)\|_{L^\infty}^2 \\
			&= \bigl( e^{(1+c_\sfF\|\rho\|_{L^\infty(L^1)})|t-s|/\e} - 1\bigr)\esssup_{t\in[0,T]}\|\sfF[\rho(t,\cdot)]-\sfF[\eta(t,\cdot)]\|_{\sup}
		\end{align*}
		for every $s,t\in[0,T]$ and $z\in\bbR^d\times\bbR^d$.
	\end{lemma}
	
	\subsection{Existence of mild solutions}
	
	We are now ready to define a candidate for our fixed point map $\Lambda\colon L^\infty((0,T); \frkR^\gamma) \to L^\infty((0,T); \frkR^\gamma)$. In the following, we consider an arbitrary but fixed initial data $f_0\in \frkX^\gamma$ with compact support $\text{\em supp} \,f_0\subset B_R(0)$. Then, we set
	\[
		(\Lambda \rho)(t,x)\coloneqq \intr (\varGamma\rho)(t,x,v)\,dv,\qquad \rho\in L^\infty((0,T);\frkR^\gamma),
	\]
	where the map $\varGamma\colon L^\infty((0,T); \frkR^\gamma) \to L^\infty([0,T];\frkX^\gamma)$ is given by
	\[
		f_\rho(t,z)\coloneqq (\varGamma \rho)(t,z) \coloneqq e^{-t/\e^2} f_0(\Phi_0^{\rho}(t,z)) + \frac{1}{\e^2}\int_0^t e^{-(t-s)/\e^2} M_\gamma[\rho_s](\Phi_s^{\rho}(t,z))\,ds,
	\]
	for $z\in \Omega_T = B(0,R_T^\rho(0))$ and $f_\rho \equiv 0$ otherwise.
	
	\medskip
	Notice that $f_\rho$ satisfies the linear kinetic equation
	\begin{align}\label{eq:lin-kin}
		\e\partial_t f_\rho + v\cdot\nabla_xf_\rho - \sfF[\rho]\cdot\nabla_z f_\rho = \frac{1}{\e}\bigl(M_\gamma[\rho] - f_\rho\bigr)\qquad\text{in the sense of distributions},
	\end{align}
	and has a bounded support that depends on the mass of $\rho$.
	
	\begin{definition}[Mild solutions]\label{def:mild-solution} A curve $f\in L^\infty([0,T];\frkX^\gamma)$ is said to be a mild solution of \eqref{eq:extend-exist} with initial data $f_0\in\frkX^\gamma$ if is satisfies
	\[
		f(t,z) = e^{-t/\e^2} f_0(\Phi_0^{\rho}(t,z)) + \frac{1}{\e^2}\int_0^t e^{-(t-s)/\e^2} M_\gamma[\rho_s](\Phi_s^{\rho}(t,z))\,ds,
	\]
	for almost every $t\in[0,T]$, where $\rho_t = \intr f(t,\cdot,v)\,dv$.
	\end{definition}
	
	\begin{theorem}\label{thm:mild}
		Let $f_0\in\frkX^\gamma$ be an initial data having compact support. Then, there exists a mild solution to \eqref{eq:extend-exist} in the sense of Definition~\ref{def:mild-solution}.
	\end{theorem}
	
	As mentioned earlier, the existence of mild solutions follows from the Leray--Schauder principle (cf. \cite[Theorem 6.A]{Zeidler}), which we recall here for completeness.
	\begin{proposition}[Schauder's fixed point theorem]\label{prop:Leray-Schauder}
		Let $S\colon X\to X$ be a compact operator on a Banach space $X$. Suppose there is an $r>0$ such that
		\[
			u = \tau S(u),\quad\text{$\tau\in(0,1)$}\quad\Longrightarrow\quad \|u\|\le r.
		\]
		Then $S$ admits a fixed point, i.e., the equation $u=S(u)$ has a solution. 
	\end{proposition}
	
	The rest of this subsection is devoted to proving that the map $\Lambda$ satisfies the assumptions of Proposition~\ref{prop:Leray-Schauder} and concluding the proof of Theorem~\ref{thm:mild}. 

	\medskip
	\paragraph{\bf\em Self-mapping property of $\Lambda$.}  
	For any $p\ge 1$, we have that
	\begin{align*}
		\|f_\rho(t,\cdot)\|_{L^p}^p &\le 2^{p-1} \left[e^{-pt/\e^2} \|f_0\|_{L^p}^p + \intrr \left(\frac{1}{\e^2}\int_0^t e^{-(t-s)/\e^2} M_\gamma[\rho_s](\Phi_s^{\rho}(t,z))\,ds\right)^p dz\right] \\
		&\le 2^{p-1} \left[e^{-pt/\e^2} \|f_0\|_{L^p}^p + \frac{1}{\e^2}\int_0^te^{-(t-s)/\e^2}\intrr   \bigl(M_\gamma[\rho_s](z)\bigr)^p dz\,ds\right].
	\end{align*}
	In the case $p=1+2/n$, we find from Lemma~\ref{lem:minimization} that
	\[
		\intrr \bigl(M_\gamma[\rho_s](z)\bigr)^{1+2/n}\,dz = 2(1+2/n)c_{\gamma,d}^{2/n} \|\rho\|_{L^\infty(L^\gamma)}^\gamma ,
	\]
	and therefore,
	\[
		\|f_\rho\|_{L^\infty(L^{1+2/n})}^{1+2/n} \le c_1\bigl( \|f_0\|_{L^{1+2/n}}^{1+2/n} + \|\rho\|_{L^\infty(L^\gamma)}^\gamma\bigr),
	\]
	for some constant $c_1>0$, depending only on $n$.
	
	Now, by construction, we have that $\text{\em supp}\, f_\rho\subset \Omega_T$. Hence, we deduce
	\[
		\|\Lambda\rho(t,\cdot)\|_{L^{1+2/n}} \le |\Omega_T^\rho|^{\frac{1}{1+n/2}}\|f_\rho\|_{L^\infty(L^{1+2/n})}.
	\]
	As for the $L^1$-norm, we obtain
	\begin{align*}
		\|(\Lambda\rho)(t,\cdot)\|_{L^1} &= \|(\varGamma\!\rho)(t,\cdot)\|_{L^1} \\
		&= e^{-t/\e^2}\|f_0\circ\Phi_0^{\rho}(t,\cdot)\|_{L^1} + \frac{1}{\e^2}\int_0^t e^{-(t-s)/\e^2} \|M_\gamma[\rho_s](\Phi_s^{\rho}(t,\cdot))\|_{L^1}\,ds \\
		&= e^{-t/\e^2}\|f_0\|_{L^1} + \frac{1}{\e^2}\int_0^t e^{-(t-s)/\e^2} \|\rho(s,\cdot)\|_{L^1}\,ds \le \|f_0\|_{L^1} + \|\rho\|_{L^\infty(L^1)}.
	\end{align*}
	Together, the interpolation inequality then yields the bound
	\[
		\|\Lambda\rho\|_{L^\infty(L^\gamma)} \le \bigl(\|f_0\|_{L^1} + \|\rho\|_{L^\infty(L^1)}\bigr)^\theta\|\Lambda\rho\|_{L^\infty(L^{1+2/n})}^{1-\theta},\qquad \theta = \frac{(\gamma-1)d}{2}.
	\]
	Altogether, we obtain a bound for $\Lambda\rho$ of the form
	\begin{align}\label{eq:Lambda-bounded}
		\|\Lambda\rho\|_{L^\infty(L^\gamma)} \le c_1^{\frac{1-\theta}{1+2/n}}|\Omega_T^\rho|^{\frac{1-\theta}{1+n/2}}\bigl(\|f_0\|_{L^1} + \|\rho\|_{L^\infty(L^1)}\bigr)^\theta\bigl( \|f_0\|_{L^{1+2/n}}^{1+2/n} + \|\rho\|_{L^\infty(L^\gamma)}^\gamma\bigr)^{\frac{1-\theta}{1+2/n}}.
	\end{align}
	
	\medskip
	\paragraph{\bf\em Continuity of $\Lambda$.} We begin by establishing the continuity of the map $\varGamma$. Since $f_0\in\frkX^\gamma$, we find a sequence $(f_0^\ell)_{\ell\in\N}\subset \calC_c(\bbR^d\times\bbR^d)$ such that $f_0^\ell\to f_0$ in $\frkX^\gamma$ as $\ell\to \infty$. Now let $(\rho^h)_{h\in\N}$ be a sequence in $L^\infty((0,T);\frkR^\gamma)$ such that $\rho^h\to \rho$ in $L^\infty((0,T);\frkR^\gamma)$. Then, for $p=1,1+2/n$, we have for almost every $t\in[0,T]$,
	\begin{align*}
		\|(\varGamma\!\rho^h -\varGamma\!\rho)(t,\cdot)\|_{L^p} &\le e^{-t/\e^2}\|f_0(\Phi_0^{\rho^h}(t,\cdot)) - f_0(\Phi_0^{\rho}(t,\cdot))\|_{L^p} \\
	&\quad+ \frac{1}{\e^2}\int_0^t e^{-(t-s)/\e^2} \|M_\gamma[\rho^h](s,\Phi_s^{\rho^h}(t,\cdot)) - M_\gamma[\rho](s,\Phi_s^{\rho}(t,\cdot))\|_{L^p}\,ds \\
	&\eqqcolon e^{-t/\e^2} I_1(t) + \frac{1}{\e^2}\int_0^t e^{-(t-s)/\e^2} I_2(t,s)\,ds.
	\end{align*}
	For $I_1$, we have that
	\begin{align*}
	I_1(t) &\le \|f_0(\Phi_0^{\rho^h}(t,\cdot)) - f_0^\ell(\Phi_0^{\rho^h}(t,\cdot))\|_{L^p} + \|f_0(\Phi_0^\rho(t,\cdot)) - f_0^\ell(\Phi_0^\rho(t,\cdot))\|_{L^p} \\
	&\qquad + \|f_0^\ell(\Phi_0^{\rho^h}(t,\cdot)) - f_0^\ell(\Phi_0^\rho(t,\cdot))\|_{L^p} \\
	&= 2\|f_0-f_0^\ell\|_{L^p} + \|f_0^\ell(\Phi_0^{\rho^h}(t,\cdot)) - f_0^\ell(\Phi_0^\rho(t,\cdot))\|_{L^p}.
	\end{align*}
Therefore, passing first $h\to\infty$ using Lemma~\ref{lem:lagrangian-stability} and then $\ell\to\infty$ yields the limit $I_1(t)\to 0$.

Similarly, we deduce
\begin{align*}
	I_2(t,s) &\le \|M_\gamma[\rho^h](s,\Phi_s^{\rho^h}(t,\cdot)) - M_\gamma[\rho^h](s,\Phi_s^{\rho}(t,\cdot))\|_{L^p} \\
	&\qquad + \|M_\gamma[\rho^h](s,\Phi_s^{\rho}(t,z)) - M_\gamma[\rho](s,\Phi_s^{\rho}(t,z))\|_{L^p} \eqqcolon I_{21}(t,s) + I_{22}(t,s),
\end{align*}
where $I_{21}$ converges to zero following the previous argument. As for the second term, we use the change of variables along the flow $\Phi_s^\rho(t,z)$ to obtain
\begin{align*}
	I_{22}(t,s) = \|M_\gamma[\rho^h](s,\cdot) - M_\gamma[\rho](s,\cdot)\|_{L^p} \to 0
\end{align*}
due to the local Lipschitz continuity of $M_\gamma$ established in Lemma~\ref{lem:nemytskii}. 

To obtain the continuity of $\Lambda$, we argue, as in the previous part, that
\[
	\|(\Lambda \rho^h - \Lambda\rho)(t,\cdot)\|_{L^p} \le |\Omega_T^\rho|^{\frac{p-1}{p}}\|(\varGamma\!\rho^h-\varGamma\!\rho)(t,\cdot)\|_{L^{p}} \longrightarrow 0
\]
for almost every $t\in(0,T)$. The interpolation inequality then yields the desired result.
	
	\medskip
	\paragraph{\bf\em Compactness of $\Lambda$.} Let $\calB\subset L^\infty((0,T);\frkR^\gamma)$ be a bounded subset. Then \eqref{eq:Lambda-bounded} shows that $\Lambda(\calB)$ is bounded. We now show that $\Lambda(\calB)$ is totally bounded. 
	
	Notice that for every $\rho\in \calB$, $f_\rho=\varGamma\!\rho\in L^\infty([0,T];\frkX^\gamma)$ satisfies \eqref{eq:lin-kin}, i.e.,
	\[
		\e\partial_t f_\rho + v\cdot \nabla_x f_\rho = g_0[\rho] + \nabla_v\cdot g_1[\rho] \qquad\text{in the sense of distributions},
	\]
	with $g_0[\rho]=\e^{-1}(M_\gamma[\rho] - f_\rho)$ and $g_1[\rho] = f_\rho\,\sfF[\rho]$, such that
	\[
		\Bigl\{ g_0[\rho] + \nabla_v\cdot g_1[\rho]: \rho\in\calB\Bigr\}\subset L^\infty((0,T);W^{-1,1+2/n}(\bbR^d\times\bbR^d))\quad\text{is bounded}.
	\]
	Together with the fact that $\{f_\rho :\rho\in\calB\}\subset L^\infty([0,T];\frkX^\gamma)$ is bounded, the standard velocity averaging lemma (cf.~\cite[Theorem~3.1]{Ja09}) implies that
	\[
		\left\{\Lambda\rho =\intr \,f_\rho(\cdot,v)\,dv:\rho\in\calB\right\} \subset \dot B_{1+2/n,\infty}^{1/(2+n)}((0,T)\times\bbR^d)\quad\text{is bounded}.
	\]
	Together with the equi-uniform integrability of $\Lambda(\calB)\subset L^1\cap L^\gamma((0,T)\times\bbR^d)$, the Fr\'echet--Kolmogorov theorem then implies the relative compactness of $\Lambda(\calB)$ in $L^1\cap L^\gamma((0,T)\times\bbR^d)$, which further implies the total boundedness of $\Lambda$ on $L^\infty((0,T);\frkR^\gamma)$.
	
	\medskip
	\paragraph{\bf\em A priori energy estimate.} 
	Finally, we show that if $\rho = \tau\Lambda(\rho)$, $\tau\in(0,1)$, then $\|\rho\|_{L^\gamma}\le r$ for some $r>0$. More precisely, we prove the following result:
	
	\begin{lemma}\label{lem:energy-inequality}
		Let $V,K\ge 0$ and $\rho = \tau\Lambda(\rho)$, $\tau\in(0,1)$. Then, $f\coloneqq\varGamma\!\rho$ satisfies
		\begin{align}\label{eq:energy-inequality}
			\calE[f(t,\cdot)] + \frac{1}{\e^2}\int_0^t\!\!\intrr \bigl(H[f](s,z) - H[M_\gamma[\rho]](s,z) \bigr)\,dz\,ds   \le \calE[f_0], 
		\end{align}
		where we recall that
		\[
			\calE[f] = \intrr H[f](z)\,dz + \intr V\rho_f\,dx + \frac12 \intr \rho_f K\ast\rho_f\,dx.
		\]
		As a consequence of the previous energy inequality, we obtain 
		\begin{align}\label{eq:density-inequality}
			\|\rho\|_{L^\infty(L^\gamma)}^\gamma \le \calE[f_0].
		\end{align}
	\end{lemma}
	\begin{proof}
	To this end, we recall the convex function $\Psi_n$ introduced in Section~\ref{sec:mini_p}:
		\[
			\Psi_n(s) = \frac{1}{2c_{\gamma,d}^{2/n}}\frac{s^{1+2/n}}{(1+2/n)}.
		\]
		In the following, we drop the subscript $n$ in $\Psi_n$ and integral domains for readability.
		
		Now set $\Psi_k(s)\coloneqq  \int_0^s k\wedge\Psi'(r)\,dr$, $k\ge 1$. It is clear to see that $\Psi_k$ is convex and satisfies $\Psi_k(s)\le ks$ for all $s\in[0,\infty)$. Since $f(t,\cdot)\in L^1(\bbR^d\times\bbR^d)$ for almost every $t\in[0,T]$, we also have that $\Psi_k(\tau f(t,\cdot))\in L^1(\bbR^d\times\bbR^d)$ for every $t\in[0,T]$ and $\tau\in(0,1)$. Moreover, the map
		\begin{align*}
			t\mapsto &\intrr \Psi_k(\tau f(t,z))\,dz \\
			&\hspace{4em}= \intrr \Psi_k\left( e^{-t/\e^2}\tau f_0(z) + \frac{1}{\e^2}\int_0^t e^{-(t-s)/\e^2} \tau M_\gamma[\rho_s](\Phi_s^{\rho}(0,z))\,ds\right)dz
		\end{align*}
		is differentiable with derivative
		\begin{align*}
			\frac{d}{dt}\intrr \Psi_k(\tau f(t,z))\,dz &= -\frac{1}{\e^2}\intrr \Psi_k'(\tau f(t,z))\,\bigr(\tau f(t,z) - \tau M_\gamma[\rho_t](z)\bigr)\,dz \\
			&\le -\frac{1}{\e^2}\intrr \Psi_k(\tau f(t,z))\,dz + \frac{1}{\e^2}\intrr \Psi_k(\tau M_\gamma[\rho_t](z))\,dz,
		\end{align*}
		where the last inequality follows from the convexity of $\Psi_k$. Similarly, we compute
		\begin{align*}
			&\frac{d}{dt} \left[\frac{1}{2}\intrr |v|^2 f(t,z)\,dz + \intr V(x) \rho(t,x)\,dx + \frac12 \intr \rho(t,x) K\ast\rho(t,x)\,dx \right]\\
			&\hspace{8em} = -\frac{1}{2\e^2}\intrr |v|^2 f(t,z)\,dz + \frac{1}{2\e^2}\intrr |v|^2 M_\gamma[\rho_t](z)\, dz.
		\end{align*}
		Together, this yields
		\begin{align*}
			&\frac{d}{dt}\left[\intrr H_k[\tau f](t,z)\,dz + \tau\intr V \rho(t,x)\,dx + \frac\tau 2 \intr \rho(t,x) K\ast\rho(t,x)\,dx \right] \\
			&\hspace{8em} \le -\frac{1}{\e^2}\intrr \bigl(H_k[\tau f](t,z) - H_k[\tau M_\gamma[\rho_t]](z) \bigr)\,dz,
		\end{align*}
		with $H_k[f] = |v|^2f/2 + \Psi_k(f)$. In particular, we obtain
		\begin{align*}
			&\intrr H_k[\tau f](t,z)\,dz + \tau\intr V(x) \rho(t,x)\,dx + \frac\tau 2 \intr \rho(t,x) K\ast\rho(t,x)\,dx \\
			&\hspace{4em}\le -\frac{1}{\e^2}\int_0^t\!\!\intrr \bigl(H_k[\tau f](s,z) - H[\tau M_\gamma[\rho_s]](z) \bigr)\,dz\,ds   \\
			&\hspace{8em}+\intrr H[f_0](z)\,dz + \intr V \rho_0\,dx + \frac12 \intr \rho_0 K\ast\rho_0\,dx,
		\end{align*}
		where we used the fact that $H_k[\tau f]\le H[\tau f]\le H[f]$ for every $k\ge 1$ and $\tau\in(0,1)$. Applying Fatou's lemma for $k\to\infty$ and $\tau\to 1$ yields the asserted energy inequality \eqref{eq:energy-inequality}.
		
		Finally, we obtain the uniform bound for $\rho\in L^\infty((0,T);\frkR^\gamma)$ by using the minimization principle (cf.\ Lemma \ref{lem:minimization}), which gives
		\begin{align*}
			\|\rho\|_{L^\infty(L^\gamma)}^\gamma &\le \esssup_{t\in(0,T)}\intrr H[M_\gamma[\rho]](t,z)\,dz \\
			&\le \esssup_{t\in(0,T)}\intrr H[f](t,z)\,dz \le \esssup_{t\in(0,T)}\calE[f(t,\cdot)] \le \calE[f_0],
		\end{align*}
		thereby concluding the proof.
		\end{proof}
		
		\paragraph{\bf\em Proof of Theorem~\ref{thm:mild}.} The proof follows by a direct application of Proposition~\ref{prop:Leray-Schauder}, since if $\rho = \Lambda(\rho)$ and $f = \varGamma\!\rho\in L^\infty([0,T];\frkX^\gamma)$, then
		\[
			f(t,z) = e^{-t/\e^2} f_0(\Phi_0^{\rho}(t,z)) + \frac{1}{\e^2}\int_0^t e^{-(t-s)/\e^2} M_\gamma[\rho_s](\Phi_s^{\rho}(t,z))\,ds,
		\]
		which shows that $f$ is a mild solution in the sense of Definition~\ref{def:mild-solution}.
		
		\subsection{Existence of weak entropy solutions}
		
		In the previous section, we obtained the existence of mild solutions satisfying the energy inequality \eqref{eq:energy-inequality} under the smoothness assumption \eqref{ass:flow} on the force $\sfF$. Here, we relax the assumption to external potentials $V$ and interaction potentials $K$ satisfying \textbf{(HV)} and \textbf{(HK)} respectively at the cost of a relaxed notion of solution. The main result of this section is thus the proof of Theorem~\ref{thm_kext}.
		
		\begin{proof}[Proof of Theorem~\ref{thm_kext}]
		We begin by considering a sequence of mollified potentials $V^\delta$ and $K^\delta$ satisfying \eqref{ass:flow} with $\nabla V^\delta \to \nabla V$ and $\nabla K^\delta\to \nabla K$ such that Theorem~\ref{thm:mild} applies, i.e., for each $\delta>0$ and a fixed $f_0\in\frkX^\gamma$ with compact support and finite energy $\calE[f_0]<+\infty$, there exists a mild solution $f^\delta\in L^\infty([0,T];\frkX^\gamma)$ to \eqref{eq:extend-exist} in the sense of Definition~\ref{def:mild-solution}. Hence, $f^\delta$ is a weak entropy solution in the sense of Definition~\ref{def_weak} due to the energy inequality deduced in Lemma~\ref{lem:energy-inequality}. In particular, 
		\[
		\e\partial_t f^\delta + v\cdot \nabla_x f^\delta = g_0^\delta + \nabla_v\cdot g_1^\delta\qquad\text{in the sense of distributions},
	\]
	with
	\[
		g_0^\delta = \frac{1}{\e} \bigl(M_\gamma[\rho^\delta] - f^\delta\bigr),\quad 
		g_1^\delta = f^\delta\,\sfF^\delta[\rho^\delta],
	\]
	where $\rho^\delta = \rho_{f^\delta}$ and $\sfF^\delta[\rho] = \nabla V^\delta + \nabla K^\delta\ast \rho$.
	
	From inequalites \eqref{eq:energy-inequality}, \eqref{eq:density-inequality}, we deduce the uniform estimate
	\[
		\sup_{\delta>0}\left\{\|\rho^\delta\|_{L^\infty([0,T];\frkR^\gamma)} +\sup_{t\in[0,T]}\left(\calE[f^\delta(t,\cdot)] +\intr |x|^2\rho^\delta(t,x)\,dx\right)\right\} <+\infty.
	\]
	Following the arguments made in Section~\ref{sec_uni}, we further obtain
	\[
		\sup_{\delta>0}\Bigl\{\|g_0^\delta\|_{L^{1+2/n}([0,T]\times\bbR^d\times\bbR^d)} + \|g_1^\delta\|_{L^{1+2/n}([0,T]\times\bbR^d\times\bbR^d)}\Bigr\} <+\infty.
	\]
	Together with a standard velocity averaging lemma (cf.~\cite[Theorem~3.1]{Ja09}), these uniform estimates allow us to conclude the existence of some $f\in L^{1+2/n}([0,T]\times\bbR^d\times\bbR^d)$ and a (not relabelled) subsequence $(f^\delta)_{\delta>0}\subset L^\infty([0,T];\frkX^\gamma)$ such that
	\begin{gather}\label{eq:exist_convergence}
		f^\delta \rightharpoonup f\;\;\text{weakly in $ L^{1+2/n}([0,T]\times\bbR^d\times\bbR^d)$}, \nonumber\\
		\rho^\delta \rightharpoonup \rho_f\;\;\text{weakly in $L^\gamma([0,T]\times\bbR^d)$},\qquad \rho^\delta \to \rho_f\;\;\text{in $L^1([0,T]\times\bbR^d)$},\\
		 \rho^\delta(t,x) \to \rho_f(t,x)\;\;\text{for a.e.\ $(t,x)\in [0,T]\times \bbR^d$}.\nonumber
	\end{gather}
	
	To obtain the energy inequality in the limit, we will need a further equicontinuity result. To do so, we first notice that for every $\varphi\in W^{2,\infty}(\bbR^d\times\bbR^d)$,
	\begin{align*}
		\sup_{\delta>0} \sup_{t\in[0,T]}|\langle v\cdot \nabla_x\varphi, f^\delta(t,\cdot)\rangle| &\le \|\varphi\|_{W^{2,\infty}}\sup_{\delta>0} \sup_{t\in[0,T]}\||v|^2 f^\delta(t,\cdot)\|_{L^1}^{1/2} \eqqcolon m_1 \|\varphi\|_{W^{2,\infty}} <+\infty,\\
		\sup_{\delta>0} \sup_{t\in[0,T]}|\langle \nabla_v\varphi, g_1^\delta(t,\cdot)\rangle| &= \sup_{\delta>0}\sup_{t\in[0,T]}|\langle \text{div}_x(\nabla_v \varphi), (V^\delta + K^\delta\ast \rho^\delta(t,\cdot))f^\delta(t,\cdot)\rangle| \\
		&\hspace{-6em}\le \|\varphi\|_{W^{2,\infty}}\sup_{\delta>0}\sup_{t\in[0,T]}\left\{\intr V^\delta(x) \rho^\delta(t,x)\,dx + \intr \rho^\delta(t,x) K^\delta\ast \rho^\delta(t,x)\,dx\right\} \\
		&\hspace{-6em} \eqqcolon m_2\|\varphi\|_{W^{2,\infty}} <+\infty.
	\end{align*}
	Hence, for any $\varphi\in W^{2,\infty}(\bbR^d\times\bbR^d)$, we have that
	\begin{align*}
		\langle\varphi,f^\delta(t,\cdot)\rangle - \langle\varphi,f^\delta(s,\cdot)\rangle &= \int_s^t \Bigl(\langle v\cdot \nabla_x\varphi, f^\delta(r,\cdot)\rangle + \langle \varphi, g_0^\delta(r,\cdot)\rangle - \langle \nabla_v\varphi,g_1^\delta(r,\cdot)\rangle \Bigr)\,dr \\
		&\le \|\varphi\|_{W^{2,\infty}}\Bigl(m_1 + \frac{1}{\e}\bigl(\|\rho_0\|_{L^1} + \|f_0\|_{L^1}\bigr) + m_2\Bigr)|t-s|.
	\end{align*}
	As a consequence, the sequence of family of curves $\{t\mapsto f^\delta(t,\cdot)\}_{\delta>0}\subset \calC([0,T];W^{-2,\infty})$ is uniformly equicontinuous. Together with the previous estimates, a generalized Arzel\'{a}--Ascoli theorem (cf. \cite{RoSa03}) may be applied to obtain the pointwise-in-time weak convergence
	\[
		f^\delta(t,\cdot)\calL^d \rightharpoonup f(t,\cdot)\calL^d\quad\text{narrowly in $\calP(\bbR^d\times\bbR^d)$ for every $t\in [0,T]$},
	\]
	along a (not relabelled) subsequence. With these convergences, we can now pass to the limit to obtain the energy inequality \eqref{kin_ineq} and the limit equation \eqref{main_eq}.
	
	\medskip
	\paragraph{\bf\em Energy inequality.} It is not difficult to see that the functional 
	\[
		\widehat{\calE}:\mathcal{P}(\bbR^d\times\bbR^d) \to [0,+\infty); \qquad \widehat{\calE}(\mu) = \begin{cases} \displaystyle
			\calE\left(\frac{d\mu}{d\calL^{2d}}\right) &\text{for $\mu \ll \calL^{2d}$,} \\
			+\infty &\text{otherwise},
		\end{cases}
	\]
	 is narrowly lower semicontinuous, and hence,
	\[
		\widehat{\calE}[f(t,\cdot)] \le \liminf_{\delta \to 0} \widehat{\calE}[f^\delta(t,\cdot)] \qquad\text{for every $t\in[0,T]$}.
	\]
	Similarly, the $L^{1+2/n}$-weak lower semicontinuity of the functional
	\[
		L^{1+2/n}((0,T)\times \bbR^d\times\bbR^d)\ni f \mapsto \iint_A\intr H[f](s,x,v)\,dv\,dx\,ds,
	\]
	for any Borel set $A\subset(0,t)\times \bbR^d$, $t\in(0,T)$, permits the liminf inequality
	\[
		\iint_A \intr H[f](s,x,v)\,dv\,dx\,ds \le \liminf_{\delta\to 0} \iint_A \intr H[f^\delta](s,x,v)\,dv\,dx\,ds.
	\]
	
	Now let $R>0$ be arbitrary. Since $\rho^\delta(t,x) \to \rho(t,x)$ for almost every $(t,x)\in [0,T]\times \overline{B_R}$, we also have that $(\rho^\delta)^\gamma(t,x) \to \rho^\gamma(t,x)$ for almost every $(t,x)\in [0,T]\times \overline{B_R}$. By Egorov's and Lusin's theorem, we find for any $\sigma>0$ a compact set $A^\sigma\subset [0,T]\times \overline{B_R}$ with $|A^\sigma|<\sigma$ such that $(\rho^\delta)^\gamma \to \rho^\gamma$ uniformly on $A^\sigma$. Consequently, have that
	\begin{align*}
		\iint_{A^\sigma}\intr H[M_\gamma[\rho^\delta]](s,x,v)\,dv\,dx\,dt &= \left(1+\frac{d}{2}\right)\iint_{A^\sigma} (\rho^\delta)^\gamma(s,x)\,dx\,ds \\
		&\hspace{-10em}\longrightarrow\; \left(1+\frac{d}{2}\right)\iint_{A^\sigma} \rho^\gamma(s,x)\,dx\,dt = \iint_{A^\sigma}\intr H[M_\gamma[\rho]](s,x,v)\,dv\,dx\,ds.
	\end{align*}
	Together with the previous convergence, we then find that
	\[
		\iint_{A^\sigma}\intr \bigl(H[f] - H[M_\gamma[\rho]] \bigr)\,dv\,dx\,ds \le \liminf_{\delta\to 0}\iint_{A^\sigma}\intr \bigl(H[f^\delta] - H[M_\gamma[\rho^\delta]] \bigr)\,dv\,dx\,ds.
	\]
	
	Moreover, due to the minimization principle (cf.\ Lemma~\ref{lem:minimization}), we also have that that
	\[
		\calE[f^\delta(t,\cdot)] + \frac{1}{\e^2}\iint_{A^\sigma} \intr\bigl(H[f^\delta] - H[M_\gamma[\rho^\delta]] \bigr)\,dv\,dx\,ds \le \calE[f_0], 
	\]
	for any $R>0$. Taking the liminf for $\delta\to 0$ then gives
	\[
		\calE[f(t,\cdot)] + \frac{1}{\e^2}\iint_{A^\sigma} \intr \bigl(H[f] - H[M_\gamma[\rho]] \bigr)\,dz\,ds   \le \calE[f_0].
	\]
	Using the minimization principle again and the monotone convergence theorem, we can first send $\sigma\to 0$ and the $R\to +\infty$ to recover the energy inequality \eqref{kin_ineq} for $f$.
	
	\medskip
	\paragraph{\bf\em Limit equation.} Recall that $f^\delta$ satisfies
			\begin{align*}
				- \intrr \varphi(0,\cdot)\,f_0\,dz - \int_0^T \!\!\! \intrr \bigl(\pa_t \varphi + v \cdot \nabla_x \varphi &- \sfF[\rho^\delta] \cdot \nabla_v \varphi \bigr) f^\delta dz\,dt \cr
				& = \int_0^T\!\!\! \intrr \varphi\lt(M_\gamma[\rho^\delta] - f^\delta\rt) dz\,dt,
			\end{align*}
	for every $\varphi \in \mc^\infty_c([0,T)\times \bbR^d \times \bbR^d)$. It is clear to see that the convergences in \eqref{eq:exist_convergence} allow us to pass to the limit in the linear terms. 
	
	As for the local equilibrium term, we notice that since $(\rho^\delta)_{\delta>0}\subset L^1\cap L^\gamma([0,T]\times\bbR^d)$ is bounded, we have from Lemma~\ref{lem:nemytskii} that $(M_\gamma[\rho^\delta])_{\delta>0}\subset L^{1+2/n}([0,T]\times\bbR^d\times\bbR^d)$ is bounded. Hence, we find a (not relabelled) subsequence and some $M_\gamma^\#\in L^{1+2/n}([0,T]\times\bbR^d\times\bbR^d)$ such that $M_\gamma[\rho^\delta] \rightharpoonup M_\gamma^\#$ weakly in $L^{1+2/n}([0,T]\times\bbR^d\times\bbR^d)$. Using the fact that $\rho^\delta\to \rho_f$ pointwise almost everywhere in $(0,T)\times\bbR^d$, which implies the pointwise almost everywhere convergence of $M_\gamma[\rho^\delta]\to M_\gamma[\rho_f]$, we then conclude that $M_\gamma^\# = M_\gamma[\rho_f]$.
	
	 We are left to prove the convergence of the interaction term, i.e.,
	\[
		\int_0^T\!\!\!\intrr \nabla K^\delta\ast\rho^\delta\cdot \nabla_v\varphi\,f^\delta\,dz\,dt \;\;\longrightarrow\;\; \int_0^T\!\!\!\intrr \nabla K\ast\rho_f\cdot \nabla_v\varphi\,f\,dz\,dt.
	\]
	To this end, we set 
	\[
		g_i^\delta(x)\coloneqq \intr \partial_{v_i}\varphi(x,v) f^\delta(x,v)\,dv,\qquad i=1,\ldots,d.
	\]
	Note that since $f^\delta(t,\cdot) \rightharpoonup f(t,\cdot)$ weakly in $L^1(\bbR^d\times\bbR^d)$ for every $t\in[0,T]$, we have that $g_i^\delta(t,\cdot) \rightharpoonup g_i(t,\cdot)$ weakly in $L^1(\bbR^d)$ for every $t\in[0,T]$. Along with the weak convergence $\rho^\delta(t,\cdot)\rightharpoonup \rho_f(t,\cdot)$ in $L^1(\bbR^d)$ for almost every $t\in[0,T]$, we find that
	\[
		g_i^\delta(t,\cdot) \otimes \rho^\delta(t,\cdot) \rightharpoonup g_i(t,\cdot)\otimes \rho_f(t,\cdot)\quad\text{weakly in $L^1(\bbR^d\times\bbR^d)$ for almost every $t\in[0,T]$},
	\]
	where $\phi\otimes \psi(x,y) = \phi(x)\psi(y)$ for measurable functions $\phi,\psi$.
	
	Now let $\varphi_i^\ell\in \calC_c(\bbR^d)$ be a sequence such that $\varphi_i^\ell\to \partial_i K$ in $(L^q + L^r)(\bbR^d)$ as $\ell\to \infty$, $i=1,\ldots,d$, which exists due to density arguments. Then,
	\begin{align*}
		&\intrr g_i^\delta(t,x) \partial_i K^\delta(x-y)\rho^\delta(t,y) \,dx\,dy - \intrr g_i(t,x) \partial_i K(x-y)\rho_f(t,y) \,dx\,dy \\
		&= \intrr \bigl[\partial_i K^\delta(x-y)-\varphi_i^\ell(x-y)\bigr]\bigl(g_i^\delta(t,x)\rho^\delta(t,y) + g_i(t,x)\rho_f(t,y)\bigr)  \,dx\,dy \\
		&\qquad+ \intrr \varphi_i^\ell(x-y) \bigl[g_i^\delta(t,x)\rho^\delta(t,y) - g_i(t,x)\rho_f(t,y)\bigr] \,dx\,dy \\
		&\eqqcolon I_1(t) + I_2(t)\qquad\text{for almost every $t\in[0,T]$}.
	\end{align*}
	Using the Young convolution inequality as in Lemma~\ref{lem:uniform-g}, we obtain the estimate
	\[
		|I_1(t)| \le C\|\partial_i K^\delta -\varphi_i^\ell\|_{L^q+L^r},
	\]
	for some constant $C>0$, independent of $\delta>0$.
	As for the second term, we use the fact that $(x,y)\mapsto\varphi_i^\ell(x-y)$ is $L^\infty$ to deduce $I_2(t)\to 0$ for almost every $t\in[0,T]$. Therefore, by, first, passing $\delta\to 0$ and then $\ell\to \infty$, we obtain $I_1(t) + I_2(t) \to 0$ for almost every $t\in[0,T]$. Finally, we apply the dominated convergence theorem to conclude that
	\[
		\int_0^T I_1(t) + I_2(t)\,dt \longrightarrow 0,
	\]
	and therewith proving that $f$ is a weak entropy solution to \eqref{main_kin} for initial data $f_0$ with compact support and finite energy $\calE[f_0]<+\infty$.
	
	To remove the compact support assumption on any initial data $f_0\in L_2^1\cap L^{1+2/n}(\bbR^d\times\bbR^d)$ with finite energy $\calE[f_0]<+\infty$, we can consider the restriction of $f_0$ on the ball $B_R(0)\subset\bbR^d\times\bbR^d$, $R>0$ and conclude using similar arguments as above when sending $R\to\infty.$
	\end{proof}
	
	%%%%%%%%%%%%%%%%%%%%%%%%%%%%%%%%%%%%%%%%%%%%%%
	%
	%
	%
	%
	%
	%
	%%%%%%%%%%%%%%%%%%%%%%%%%%%%%%%%%%%%%%%%%%%%%%
	
	\section*{Acknowledgement}
	The research of YPC was supported by NRF grant no.\ 2022R1A2C1002820 and RS-2024-00406821. J. Kim is grateful for support by the Open KIAS Center at Korea Institute for Advanced Study.
	
	%%%%%%%%%%%%%%%%%%%%%%%%%%%%%%%%%%%%%%%%%%%%%%
	%
	%
	%
	%
	%
	%
	%%%%%%%%%%%%%%%%%%%%%%%%%%%%%%%%%%%%%%%%%%%%%%
 	
 	\appendix
 	\section{Proof of Lemma~\ref{lem:averaging}}\label{sec:velocity}

	%\begin{proof}[Proof of Lemma~\ref{lem:averaging}]
			Let $f,g_0,g_1\in L^p((0,T);L_{loc}^p(\bbR^d\times\bbR^d))$. Consider a compactly supported smooth function $\widetilde{\varphi}\in \calC_c^\infty(\bbR^d\times\bbR^d)$ and set
			\[
				\widetilde{f}(t,x,v) \coloneqq f(t,x,v)\,\widetilde{\varphi}(x,v).
			\]
			Then, the function $\tilde{f}$ satisfies
			\[
			\e\pa_t\widetilde{f} + v\cdot\nabla_x\widetilde{f} = \widetilde{g}_0+\nabla_v\cdot\widetilde{g}_1,
			\]
			where 
			\[
			\widetilde{g}_0 \coloneqq g_0 \widetilde{\varphi} + f(v\cdot\nabla_x\widetilde{\varphi})-g_1\cdot\nabla_v\widetilde{\varphi}\qquad \mbox{and} \qquad \widetilde{g}_1\coloneqq g_1\widetilde{\varphi}.
			\]
			Thus, we obtain $\tilde{f},\tilde{g}_0,\tilde{g}_1\in L^p((0,T);L^p(\bbR^d\times\bbR^d))$, implying that once we prove Lemma \ref{lem:averaging} for the $L^p$ case, we conclude that for other compactly supported smooth functions $\phi,\psi\in \calC^\infty_c(\bbR^d)$,
			\begin{align*}
			\intr f(t,x,v)\,\widetilde{\varphi}(x,v)\phi(x)\psi(v)\,dv = \intr\widetilde{f}(t,x,v)\phi(x)\psi(v)\,dv \in L^{p}\bigl((0,T);B^{(p-1)/2p}_{p,\infty}(\bbR^d)\bigr).
			\end{align*}
			Now, by choosing $\widetilde{\varphi}$ as a smooth cutoff function such that
			\[
				\widetilde{\varphi}(x,v)=
			1 \quad \mbox{for $(x,v)\in (\textup{supp}\,\phi)\times (\textup{supp}\,\psi)$},
			\]
			we obtain $\widetilde{\varphi}(x,v)\phi(x)\psi(v)=\phi(x)\psi(v)$ for $(x,v)\in ({\supp}\phi)\times({\supp} \psi)$, thus implying
			\[
				(\rho_{\psi}\phi)(t,x)=\intrr f(t,x,v)\phi(x)\psi(v)\,dv\in L^{p}\bigl((0,T);B^{(p-1)/2p}_{p,\infty}(\bbR^d)\bigr).
			\]
			Thus, it remains to prove the desired assertion for $f,g_0,g_1\in L^p((0,T)\times\bbR^d\times\bbR^d)$.
			
			For any $\lambda>0$, we add $\lambda f$ to both sides to deduce
			\[
				\lambda f +\e\pa_t f + v\cdot\nabla_x f = g_0+\nabla_v\cdot g_1 + \lambda f.
			\]
			From this, we obtain
			\[
				\rho_\psi(t,x) =\intr f(t,x,v)\,\psi(v)\,dv=T_\lambda (g_0+\nabla_v \cdot g_1 +\lambda f),
			\]
			where the operator $T_\lambda$ is defined as
			\[
				(T_\lambda h)(t,x) \coloneqq \int_0^\infty\!\!\! \intr h(t-\e \tau,x-\tau v,v)e^{-\lambda \tau}\psi(v)\,dv\,d\tau.
			\]
			Then, we use \cite[Proposition 3.1]{JV04} with the parameters $s=0$, $p_1=p_2=p$ to deduce
			\[
				\lambda^{\frac{1}{p}}\|T_\lambda(h)\|_{L^p_t\dot{W}^{1-\frac{1}{p},p}_x}\le C_\psi\|h\|_{L^p(\R\times\bbR^d\times\bbR^d)}.
			\]
			Using the same proposition with the parameters $s=-1$, $p_1=p_2=p$, we have
			\[
				\lambda^{1+\frac{1}{p}}\|T_{\lambda}(h)\|_{L^p_t\dot{W}^{-\frac{1}{p},p}_x} + \lambda^{\frac{1}{p}}\|T_\lambda(h)\|_{L^p_t\dot{W}^{1-\frac{1}{p},p}_x}\le C_\psi\|h\|_{L^p(\R_t\times\R_x^d;W^{-1,p}(\R_v^d))}.
			\]
            Here, $C_\psi = C \|\psi\|_{L^1 \cap W^{1,\infty}}$ for some universal constant $C>0$.
			Note that we cannot obtain an additional regularity in the time variable due to the coefficient $\e$ in front of $\pa_t f$ in \eqref{eq:vel-avg}.
			Thus, once we split $\rho_\psi$ as
			\[
				\rho_\psi = \lambda T_{\lambda} f + T_{\lambda}(g_0+\nabla_v\cdot g_1) = \rho^1_\psi +\rho^2_\psi,
			\]
			we have
			\[
				\lambda^{-1+\frac{1}{p}}\|\rho^1_\psi\|_{L^p_t\dot{W}^{1-\frac{1}{p},p}_x}\le C_\psi\|f\|_{L^p},
			\]
			and
			\[
				\lambda^{1+\frac{1}{p}}\|\rho^2_\psi\|_{L^p_t\dot{W}^{-\frac{1}{p},p}_x} + \lambda^{\frac{1}{p}}\|\rho^2_\psi\|_{L^p_t\dot{W}^{1-\frac{1}{p},p}_x}\le C_\psi\|g_0+\nabla_v\cdot g_1\|_{L^p}.
			\]
			Therefore, after localizing $\rho_\psi$ by multiplying $\phi$, we can deduce that $\rho_\psi\phi=\rho_\psi^1\phi+\rho_\psi^2\phi$,
			satisfying
			\[
				\|\rho_\psi^1\phi\|_{L^p_t \dot W^{1-\frac{1}{p},p}_x}\le C_\psi C_\phi\lambda^{1-\frac{1}{p}} \quad \mbox{and} \quad \|\rho_\psi^2\phi\|_{L^p_t \dot W^{-\frac{1}{p},p}_x}\le C_\psi C_\phi\lambda^{-1-\frac{1}{p}},
			\]
where $C_\phi = C\|\phi\|_{ W^{1,\infty}}$ for some universal constant $C>0$. To establish that $\rho_\psi\phi$ belongs to the desired Besov space, we apply the $K$-method for real interpolation \cite{BL76,JV04}. Specifically, we define
		\[
			K(h)=\inf_{a_1+a_2=\rho_\psi\phi}\left\{\|a_1\|_{L^p_t \dot W^{1-\frac{1}{p},p}_x}+h\|a_2\|_{L^p_t \dot W^{-\frac{1}{p},p}_x}\right\}.
		\]
For each $h>0$, we set $\lambda = h^{1/2}$. Choosing $a_1=\rho_\psi^1\phi$ and $a_2=\rho_\psi^2\phi$, we obtain
		\[
			K(t)\le C_\psi C_\phi\lambda^{1-\frac{1}{p}}=C_\psi C_\phi h^{\frac{p-1}{2p}},\quad h>0.
		\]
		This implies
		\[
			\rho_\psi \phi\in \left(L^p_t \dot W^{1-\frac{1}{p},p}_x,L^p_t \dot W^{-\frac{1}{p},p}_x\right)_{s,\infty}=L^p_t \dot B^{s}_{p,\infty},
		\]
			where
			\[s=\left(1-\frac{1}{p}\right)\left(\frac{p+1}{2p}\right)+\left(-\frac{1}{p}\right)\left(\frac{p-1}{2p}\right)=\frac{p-1}{2p}.\]
			In conclusion, for $p \in (1, 1 + \frac 2n]$, we have
		\[
		\rho_{\psi}\phi\in L^p\bigl((0,T);B^{(p-1)/2p}_{p,\infty}(\bbR^d)\bigr) \qquad\text{for any $\phi,\psi\in \calC^\infty_c(\bbR^d)$}.\qedhere
		\]

 	%%%%%%%%%%%%%%%%%%%%%%%%%%%%%%%%%%%%%%%%%%%%%%
	%
	%
	%
	%
	%
	%
	%%%%%%%%%%%%%%%%%%%%%%%%%%%%%%%%%%%%%%%%%%%%%%

\end{document}